\documentclass[11pt,reqno]{amsart}

\usepackage{amssymb}
\usepackage{latexsym}
\usepackage{amsbsy}
\usepackage{amsfonts}
\usepackage{appendix}
\usepackage{pstricks}
\usepackage{pst-fill,pst-grad}
\usepackage{color}

\usepackage{enumitem}

\numberwithin{equation}{section}
\newtheorem{theorem}{Theorem}[section]
\newtheorem{corollary}[theorem]{Corollary}
\newtheorem{lemma}[theorem]{Lemma}
\newtheorem{proposition}[theorem]{Proposition}

\theoremstyle{definition}
\newtheorem{definition}[theorem]{Definition}
\theoremstyle{remark}

\newtheorem{remark}[theorem]{Remark}

\newdimen\AAdi%
\newbox\AAbo%
\def\AAk#1#2{\setbox\AAbo=\hbox{#2}\AAdi=\wd\AAbo\kern#1\AAdi{}}%

\makeatletter
\def\eqlabel#1{\def\@currentlabel{#1}}

\def\formula#1{\def\@tempa{#1}\let\@tempb\theequation\def\theequation{%
\hbox{#1}}\def\@currentlabel{(\theequation)}$$}
\def\endformula{\leqno\hbox{(\@tempa)}$$\@ignoretrue\let\theequation\@tempb}

\def\given{\hskip5\p@\relax\vrule\@width.4\p@\hskip5\p@\relax}

\newcommand{\open}[1]{%
\par\normalfont\topsep6\p@\@plus6\p@\trivlist\item[\hskip\labelsep\itshape#1%
\@addpunct{.}]\ignorespaces}

\DeclareRobustCommand{\close}[1]{%
  \ifmmode 
  \else \leavevmode\unskip\penalty9999 \hbox{}\nobreak\hfill
  \fi
  \quad\hbox{$#1$}}
\makeatother

\newlength{\toskip}

\def\<{\langle}
\def\>{\rangle}

\def \R {{\mathbb R}}

\def \L {{\mathbb L}}

\def \Var {\textrm{Var}}
\def \Ent {\textrm{Ent}}
\def \Osc {\textrm{Osc}}

\newcommand{\Hess}{\operatorname{Hess}}
\newcommand{\Ric}{\operatorname{Ric}}

\begin{document}
\date{\today}


\title[Functional inequalities with variable curvature]{Self-improvement of the Bakry-Emery criterion for Poincar\'e inequalities and Wasserstein contraction using variable curvature bounds}

 \author[P. Cattiaux]{\textbf{\quad {Patrick} Cattiaux $^{\spadesuit}$ \, \, }}
\address{{\bf {Patrick} CATTIAUX},\\ Institut de Math\'ematiques de Toulouse. CNRS UMR 5219. \\
Universit\'e Paul Sabatier,
\\ 118 route
de Narbonne, F-31062 Toulouse cedex 09.} \email{patrick.cattiaux@math.univ-toulouse.fr}

 \author[M. Fathi]{\textbf{\quad {Max} Fathi $^{\spadesuit \clubsuit}$ \, \, }}
\address{{\bf {Max} FATHI},\\ CNRS and Institut de Math\'ematiques de Toulouse. CNRS UMR 5219. \\
Universit\'e Paul Sabatier,
\\ 118 route
de Narbonne, F-31062 Toulouse cedex 09.} \email{max.fathi@math.univ-toulouse.fr}

\author[A. Guillin]{\textbf{\quad {Arnaud} Guillin $^{\diamondsuit}$}}
\address{{\bf {Arnaud} GUILLIN},\\ Laboratoire de Math\'ematiques, CNRS UMR 6620, Universit\'e Clermont-Auvergne,
avenue des Landais, F-63177 Aubi\`ere.} \email{arnaud.guillin@uca.fr}

\maketitle

 \begin{center}

 \textsc{$^{\spadesuit}$  Universit\'e de Toulouse}
\smallskip

 \textsc{$^{\clubsuit}$  CNRS}
\smallskip

\textsc{$^{\diamondsuit}$ Universit\'e Clermont-Auvergne}
\smallskip

\end{center}

\begin{abstract}
 We study Poincar\'e inequalities and long-time behavior for diffusion processes on $\R^n$ under a variable curvature lower bound, in the sense of Bakry-Emery. We derive various estimates on the rate of convergence to equilibrium in $L^1$ optimal transport distance, as well as bounds on the constant in the Poincar\'e inequality in several situations of interest, including some where curvature may be negative. In particular, we prove a self-improvement of the Bakry-Emery estimate for Poincar\'e inequalities when curvature is positive but not constant. 
\end{abstract}
\bigskip

\bigskip

\bigskip

\section{Introduction.}

Let $\mu(dx) = Z^{-1} \, e^{-V(x)} \, dx$ be a probability measure defined on $\mathbb R^n$. We assume that $V$ is a smooth enough function (of $C^\infty$ class in this introduction). The most classical consequence of the celebrated Bakry-Emery criterion $CD(\rho,\infty)$ (see \cite[Definition 1.16.1]{BaGLbook}) is that, as soon as the Hessian of $V$ is uniformly positive definite, i.e. for some $\rho>0$, all $u \in \mathbb R^n$, all $x \in \mathbb R^n$
\begin{equation}\label{eqBE}
\langle u \, , \, \Hess V(x) \, u\rangle \, \geq \, \rho \, |u|^2 \, ,
\end{equation}
then $\mu$ satisfies several functional inequalities, including the following:
\begin{definition}
A probability measure $\mu$ satisfies a Poincar\'e inequality if for all smooth $f$, (here and in the sequel we denote by $\mu(f)$ the integral of $f$ w.r.t. $\mu$),
\begin{equation}\label{eqpoinc}
\Var_\mu(f):= \mu(f^2) - \mu^2(f) \leq C_P(\mu) \, \mu(|\nabla f|^2) \, .
\end{equation}
It satisfies a logarithmic Sobolev inequality (LSI) if for all smooth $f$
\begin{equation}\label{eqLS}
H(f^2|\mu):= \mu(f^2 \, \ln (f^2)) - \mu(f^2) \, \ln(\mu(f^2)) \leq C_{LS}(\mu) \, \mu(|\nabla f|^2) \, .
\end{equation}
\end{definition}
%

In the sequel, $C_P$ and $C_{LS}$ should be understood as the best constants for which the previous inequalities hold, for a given probability measure. We refer to \cite{Ane,BaGLbook,Wbook} among many others, for a comprehensive introduction to some of the useful consequences of these inequalities, as well as their most important properties.
\medskip

When curvature is bounded from below by some constant $\rho > 0$, then the celebrated Bakry-Emery theorem states that 
\begin{equation} \label{be_bnds}
C_P \leq \rho^{-1}; \hspace{5mm} C_{LSI} \leq 2\rho^{-1}.
\end{equation}
An interesting remark is that in the Gaussian case, the Poincar\'e and logarithmic Sobolev constants obtained through this (seemingly) crude upper bound are in fact optimal. Under this curvature condition, the key element in the usual proofs of these functional inequalities is that the associated semi-group $P_t=e^{tL}$, where $L=\Delta -\nabla V.\nabla$, satisfies (provided it is well defined), as soon as \eqref{eqBE} holds, the pseudo commutation property
\begin{equation}\label{eqBEcommut}
|\nabla P_t f | \, \leq \, e^{-\rho t} \, P_t(|\nabla f|)
\end{equation}
for all smooth $f$. The constant $\rho$ can be seen as some kind of Ricci curvature lower bound for the semi-group. This terminology comes form the fact that, for a Brownian motion on a smooth manifold, this property is actually equivalent to having a lower bound of the form $\Ric \geq \rho g$, where $\Ric$ is the Ricci curvature tensor and $g$ the metric tensor. Actually \eqref{eqBEcommut} is true as soon as \eqref{eqBE} is satisfied even $\rho \in \mathbb R$ is non-positive, but in this case one cannot immediately deduce the inequalities we are interested in for $\mu$. The probabilistic interpretation is provided by the associated stochastic process $X_t^x$ solution of the stochastic differential (integral) equation
\begin{equation}\label{eqsde}
X_t^x \, = \, x + \sqrt 2 \, B_t \, - \, \int_0^t \, \nabla V(X_s^x) \, ds
\end{equation}
where $B_.$ is a standard Brownian motion, and with the semi group then satisfying the formula 
\begin{equation}\label{eqsemi}
P_t f(x) \, = \, \mathbb E(f(X_t^x)) \, := \, \mathbb E_x(f(X_t)) \, .
\end{equation}
The notation $\mathbb E_\nu$ will be used when looking at the stochastic process starting from some random initial data $X_0$ with distribution $\nu$. One can check that $\mu$ is a stationary (and even reversible) measure for this process.
\medskip

A natural question is to extend \eqref{eqBEcommut} to the more general situation of a non-constant Ricci curvature lower bound. More precisely, since $\Hess V(x)$ is a real symmetric matrix, all its eigenvalues $(\rho_i(x))_{i=1,...,n}$ are real and 
\begin{equation}\label{eqBEvar}
\langle u \, , \, Hess V(x) \, u\rangle \, \geq \, \rho(x) \, |u|^2 \, , \textrm{with $\rho(x)=\rho_m(x)=\min_{i=1,...,n} \, \rho_i(x)$.}
\end{equation}
Notice that $\rho_m(x)$ is optimal in the previous inequality simply looking at the eigenvectors of the Hessian matrix. 
\smallskip

\begin{definition}\label{defcurve}
We shall say that the curvature is bounded from below if \eqref{eqBEvar} is satisfied for some $\rho(x) \geq \rho_0 \in \mathbb R$ and any $x$. 
\end{definition}
Equivalently, for $\lambda> - \, \rho_0/2$, $\lambda |x|^2 + V(x)$ is uniformly convex. It follows that, defining $\psi(x)=|x|^2$, $L\psi (x) \leq K \psi(x) + a^2$ for some constant $a$. It is then classical that the process is conservative (non explosive) and ergodic with invariant (reversible) measure $\mu$. More precisely, thanks to reversibility (i.e. symmetry in $L^2(\mu)$), the spectral theorem allows us to show that for $f \in \mathbb L^2(\mu)$, $P_t f \to \mu(f)$ as $t \to +\infty$ in $\mathbb L^2(\mu)$ (a direct probabilistic proof is contained in \cite{CatPota}).
\medskip

One motivation for introducing a variable curvature bound, even in situations where curvature is actually bounded from below by some positive constant, is that the Bakry-Emery theorem is rigid \cite{CZ17,GKKO,OT19}: if equality holds in either bounds of \eqref{be_bnds}, then the measure \emph{must} split off a Gaussian factor, and therefore the optimal $\rho_m$ in \eqref{eqBEvar} is constantly equal to $\rho$. Therefore for any measure satisfying a Ricci curvature bound with a \emph{variable} optimal curvature lower bound must satisfy a Poincar\'e inequality and a logarithmic Sobolev inequality with strictly better bounds. An improvement using the harmonic average of the curvature bound was obtained in \cite{Vey11} for the Poincar\'e constant. See also \cite{DPF17, CF18} for some results in that direction, of a different flavor than those we shall obtain here. 

We shall first give a simple proof of the following generalization of \eqref{eqBEcommut} (see section \ref{seccouple} for precise statements)
\begin{proposition}\label{propBEvarcommut}
If \eqref{eqBEvar} is satisfied for some regular enough function $\rho$, it holds $$|\nabla P_t f|(x)  \, \leq \, \mathbb E_x \left(e^{-\int_0^t \, \rho(X_s) \, ds} \, |\nabla f|(X_t)\right) \, .$$
\end{proposition}
This result is not new. It is shown for instance in the recent work by Braun, Habermann and Sturm \cite{BHS} in the much more general framework of metric measure spaces. That work also contains numerous results on equivalences between several notions of ``variable Ricci curvature'', and was actually the starting point of the present work. In the diffusion case we are looking at, some of us heard about Proposition \ref{propBEvarcommut} from D. Bakry in informal discussions. A partial result in this direction is also contained in \cite[Section 7.4]{CGsemin}. Some related gradient bounds using local curvature bounds were investigated in \cite{ATW14}.

\noindent The natural question is then to understand what consequences we can get from this gradient bound. The goal of the present work is to derive some applications to the study of the long-time behavior of the underlying stochastic dynamic.
\medskip

We would like to mention another related approach, the so called ``intertwining'' method. When $\mu$ is gaussian with covariance matrix $\rho Id$ (i.e. $X_.$ is an Ornstein-Uhlenbeck process), the inequality \eqref{eqBEcommut} becomes an equality 
 $$\nabla P_t f(x) = e^{-\rho t} P_t(\nabla f)(x) \, .$$ An intertwining semi-group is a distortion of the gradient such that $$\nabla P_t f(x) = P_t^A(A \, \nabla f)(x) \, ,$$ for some perturbed semi-group $P_t^A$. In some situations, if the perturbation is nice enough one may use it to recover gradient bounds similar to \eqref{eqBEcommut}. The intertwining method has been well known for a long time in the context of stochastic processes.  Its application to functional inequalities (Poincar\'e or Brascamp-Lieb) is more recent. For background about implementation of the intertwining method for diffusion processes, we refer to \cite{ABJ,BJ,BJM}.
\medskip

Another approach that can sometimes be applied in situations where curvature may be negative is that of F-Y. Wang \cite{W97}, later extended by E. Milman \cite{emil2}, who showed that if curvature is bounded from below by some uniform, but possibly negative constant, a strong enough Gaussian concentration inequality allows to recover a logarithmic Sobolev inequality, and even an isoperimetric inequality. See also \cite{GRS} for an alternative proof, and \cite{barmil} for some applications in statistical physics. The results obtained via that method and those we shall present here do not seem to be directly comparable. 

In order to use Proposition \ref{propBEvarcommut}, consider some $1$-Lipschitz function $f$. If $X_0$ is distributed according to $\nu$, we denote by $P_t^*\nu$ the distribution of $X_t$. Since the semi-group is symmetric we should omit the $^*$, but it helps to understand the nature of the various objects. We thus have for probability measures $\nu$ and $\beta$, $$|P_t^*\nu(f) - P_t^*\beta(f)| = |\nu(P_t f)-\beta(P_t f)|$$ and $$\parallel |\nabla P_t f|\parallel_\infty \, \leq \, \sup_x \mathbb E_x\left[e^{-\int_0^t \, \rho(X_s) \, ds}\right] \, .$$ It follows, using its variational expression, that the $W_1$ Wasserstein distance satisfies
\begin{equation}\label{eqw1-1}
W_1(P_t^*\nu,P_t^*\beta) \, \leq \, \sup_x \mathbb E_x\left[e^{-\int_0^t \, \rho(X_s) \, ds}\right] \, W_1(\nu,\beta) \, .
\end{equation}
In order to get some decay to $0$ for $W_1(P_t^*\nu,\mu)$ , it thus remains to estimate the sup-norm of $\mathbb E_x \left(e^{-\int_0^t \, \rho(X_s) \, ds}\right)$. This connection between convergence in Wasserstein distance and gradient estimates is a particular instance of the Kuwada duality theorem \cite{Kuw10}.
\smallskip

Assuming that the process is ergodic (for instance when the curvature is bounded from below), we know that 
\begin{equation}\label{eqergod}
\frac 1t \, \int_0^t \, \rho(X_s)  \, ds \; \to \; \mu(\rho)
\end{equation}
as $t \to \infty$, $\mathbb P_x$ almost surely for all $x$. 
\smallskip

Hence we should expect that $e^{-\int_0^t \, \rho(X_s) \, ds}$ may behave like $e^{-\mu(\rho)t}$ for large $t$, so that one can expect that, replacing $\rho = \min_x \rho(x)$ in the Bakry-Emery criterion by $\mu(\rho)$, will allow us to derive interesting results as soon as $\mu(\rho)>0$. In particular, we can hope to handle some situations where $\rho$ is negative in some region of space, possibly even at infinity.  
\smallskip

\noindent It turns out that $\mu(\rho_m)$ is often positive. For instance if $n=1$, since $\rho_m=V''$, integrating by parts yields, 
$$\mu(\rho_m) = \int \, V'' \, e^{-V} dx = \int \, (V')^2 \, e^{-V} \, dx \, > \, 0 \, ,$$ provided the integrals exist and $V' e^{-V}$ goes to $0$ at infinity. 

\noindent For a general $n$, and any $u \in \mathbb R^n$, it holds $$\int \, \langle u \, , \, \Hess V \, u\rangle \, e^{-V} \, dx \, = \, \int \, \langle u \, , \, \nabla V\rangle^2 \, e^{-V} \, dx \, \geq 0 \, ,$$ provided everything makes sense and $\langle u \, , \, \nabla V\rangle \, e^{-V}$ goes to $0$ at infinity. But in order to estimate $\mu(\rho_m)$ we would have to take the infimum w.r.t. $u$ under the integral sign, so that (strict) positivity is unclear.

\noindent Nevertheless, the optimal lower bound on average curvature is often positive. This means that something else is needed in order to obtain interesting consequences, since the asymptotic behavior of the exponential term in \eqref{eqw1-1} cannot be that simple. 
\smallskip

Another reason is the following: if we have a uniform contractive estimate of the form 
$$\sup_x \mathbb E_x\left(e^{-\int_0^t \, \rho(X_s) \, ds}\right) \, \leq \, e^{-ct} \, ,$$ 
then the semi-group is a contraction in $W_1$ distance. This property is known to imply that the curvature is bounded from below by a \emph{positive} constant \cite{RS}.
\medskip

One weakness in this approach is that \eqref{eqw1-1} involves a supremum over $x$. Such a uniform control requires strong continuity assumptions for the semi-group (for instance ultra-boundedness). It is thus interesting to try to get a direct expression for $W_1(P^*_t \nu \, , \, P^*_t \beta)$ without using the variational expression of $W_1$. This is done in some specific situations in Section \ref{secWcouple}.
\medskip

Similarly, one can try to get $\mathbb L^2$ estimates. Recall that $$\Var_\mu(P_tf) \, = \, 2 \,  \int_t^{+\infty} \, \mu(|\nabla P_sf|^2) \, ds \, .$$ Hence if $f$ is $1$-Lipschitz we thus have 
\begin{equation}\label{eqvardecay}
\Var_\mu(P_tf) \, \leq \, 2 \, \int_t^{+\infty} \, \mathbb E_\mu\left(e^{-\int_0^s \, 2 \, \rho(X_u) \, du}\right) \, ds \, .
\end{equation}
 If we are able to show that $$\mathbb E_\mu\left(e^{-\int_0^s \, 2 \, \rho(X_u) \, du}\right) \, \leq \, C \, e^{-cs}$$ for some $c>0$ provided $\mu(\rho)>0$, one can expect some decay $$ \Var_\mu(P_tf) \, \leq \, \frac{2C}{c} \, e^{-ct}$$ for all $1$-Lipschitz function $f$. Recall the following result \cite[Lemma 2.12]{CGZ}: 
 \begin{lemma}\label{lemdecayf}
  Let $\mathcal{C}$ be a dense subset of $\mathbb L^2(\mu)$. Suppose that there exists $c>0$, and, for any $f \in \mathcal{C}$, a constant $c_f$ such that:
$$ \forall t, \quad  \Var_\mu(P_t f) \leq c_f \; e^{- \, c t} \, .$$ Then
$$\forall f \in \L^2(\mu), \forall t, \quad \Var_\mu (P_t f) \leq e^{- \, c t}\Var_\mu(f) \, .$$
\end{lemma}
Using homogeneity and since we are in the reversible situation, we will deduce that for all $f$, $\Var_\mu(P_tf) \leq e^{-ct} \, \Var_\mu(f)$, so that $\mu$ satisfies a Poincar\'e inequality with constant less than $2/c$.
\medskip

In order to develop this strategy, what is required is thus to control the rate of convergence in \eqref{eqergod}. Deviation bounds for additive functionals of ergodic diffusion processes have been studied in \cite{CGesaim,GLWY} based on a previous result by Wu (\cite{Wu0}). One can also mention \cite{Lez} where analytic tools are used. We shall mainly use \cite{GLWY} where a very detailed study is performed. It is also presumably possible to use direct controls as it is done in \cite{MT} for the one dimensional situations, though some points in \cite{MT} are not totally clear for us.
\medskip

To end this introduction let us give a flavor of the results we obtained in the case of Poincar\'e constant (see  Th.\ref{BEimpliPoinc} for a more complete result)
\begin{theorem}
Assume that $V$ is $C^2$ and that $\rho$ is bounded below by\ $\rho_0 >0$, then $$C_P(\mu) \leq \frac{1}{\rho_0 + \varepsilon'}$$ with 

$$\varepsilon' =  \left((\mu(\rho)-\rho_0) + \, \frac{1}{\rho_0} \, \Osc^2(\rho)\right) \left(1 \, - \, \sqrt{1 \, - \, \frac{(\mu(\rho)- \rho_0)^2}{((\mu(\rho)-\rho_0)+ \, \frac{1}{\rho_0}  \, \Osc^2(\rho))^2}}\right) \, .$$
\end{theorem}
Remark that we may always choos $\rho$ to be bounded, eventually loosing on $\mu(\rho)$. We thus obtain a strict improvement of the Bakry-Emery estimation of the Poincar\'e constant in the variable curvature case.
\medskip

The plan of the paper is the following: in Section \ref{seccouple} we prove by differentiation of the flow or by coupling  Proposition \ref{propBEvarcommut}. Then we introduce the methodology for the Wasserstein contraction in Section 3 showing that the crucial estimates is the one of $E_\nu\left(e^{-\int_0^t \, \rho(X_s) \, ds}\right)$. By using transport-information inequalities we obtain controls for such Laplace transform in Section 4, where a particular attention is considered for the effect of the initial measure $\nu$, thus obtaining Wasserstein contraction estimates as well as general bounds on the Poincar\'e constant. Finally, Section 5 considers various examples and applications, e.g. the logconcave case or the uniformly convex case.

\section{Proof(s) of Proposition \ref{propBEvarcommut}.}\label{seccouple}

We shall propose two approaches for proving Proposition \ref{propBEvarcommut}: the first one is based on the theory of stochastic flows in the spirit of \cite{Kunita,IW}, the second one on coupling following the ideas in \cite{CGsemin, Vey11}. Since the framework is simpler than that of \cite{BHS}, the proofs are more direct. 
\smallskip

\subsection{Proof via stochastic flows.\\ \\}\label{secflow}
First recall the following result which is a consequence of \cite[Theorem 3.1 and Theorem 5.4]{Kunita}
\begin{theorem}\label{thmkun}
Let $$X^x_t \, = \, x - \int_0^t \, \nabla V(X_s^x) \, ds \, + \, \sqrt 2 \, B_t \, .$$ Assume that $V$ is $C^k$, and the process is conservative (i.e. for all $x$ the lifetime is almost surely infinite). Then for all $t$ the application $x \mapsto X_t^x$ is a.s. $C^{k-2}$, and its derivatives are obtained by formal differentiation. In particular $$\partial_i X_t^x:=\frac{\partial X_t^x}{\partial x_i}$$ satisfies $$\partial_i X_t^x \, = \, e_i \, - \, \int_0^t \, \Hess V(X_s^x) \, \partial_i X_s^x \, ds \, ,$$ where $e_i$ is the $i^{th}$ canonical unit vector, i.e. $$\partial_i X_t^x \, = \, \left(\exp \, \left(- \, \int_0^t \, \Hess V(X_s^x) \, ds\right)\right) \, e_i \, .$$
\end{theorem}
If $f$ is a smooth function, we thus have $$\frac{\partial}{\partial x_i} \,  (f(X_t^x)) = \left\langle (\nabla f)(X_t^x) \, , \, \left(\exp \, \left(- \, \int_0^t \, \Hess V(X_s^x) \, ds\right)\right) \, e_i\right\rangle$$ so that $$|\nabla (f(X_t^x))|^2 = \left\langle (\nabla f)(X_t^x) \, , \, \left(\exp \, \left(- \, 2 \, \int_0^t \, \Hess V(X_s^x) \, ds\right)\right) (\nabla f)(X_t^x)\right\rangle \, .$$ Since $P_tf(x)=\mathbb E(f(X_t^x))$ we may calculate $\nabla P_tf(x)$ by differentiating under the expectation, if such a differentiation is allowed. But from what precedes
\begin{eqnarray}\label{eqhessalmost}
|\nabla (f(X_t^x))| &\leq& |(\nabla f)(X_t^x)| \; \sup_{u \in \mathbb S^{n-1}} \, \left\langle u \; , \; \left(\exp \, \left(- \, 2 \, \int_0^t \, \Hess V(X_s^x) \, ds\right)\right) \, u\right\rangle^{\frac 12} \\ &\leq& e^{- \, \left(\int_0^t \, \rho_m(X_s^x) \, ds\right)} \, |(\nabla f)(X_t^x)| \, , \nonumber
\end{eqnarray}
where $\rho_m$ is defined in \eqref{eqBEvar}. 

Indeed, if $\lambda(t,x)$ denotes the smallest eigenvalue of the real and symmetric matrix $\int_0^t \, Hess V(X_s^x) \, ds$ the supremum in \eqref{eqhessalmost} is attained for a corresponding normalized eigenvector $v(t,x)$ and is equal to $e^{- \, \lambda(t,x)}$. But
\begin{eqnarray*}
\lambda(t,x) &=& \left\langle v(t,x) \; , \; \left(\int_0^t \, \Hess V(X_s^x) \, ds\right) \, v(t,x)\right\rangle \\ &=& \int_0^t \, \langle v(t,x) \; , \;  \Hess V(X_s^x)  \, v(t,x)\rangle  \, ds \\ &\geq&  \left(\int_0^t \, \rho_m(X_s^x) \, ds\right) \, |v(t,x)|^2 \, = \, \int_0^t \, \rho_m(X_s^x) \, ds \, .
\end{eqnarray*}

Hence using Lebesgue's theorem for differentiating under the expectation we have obtained the following precise version of Proposition \ref{propBEvarcommut}
\begin{theorem}\label{thmBEvarcommut}
Assume that $V$ is $C^3$, that the process is conservative and that $x \mapsto \mathbb E\left(e^{-\int_0^t \, \rho(X_s^x) \, ds}\right)$ is locally bounded, $\rho$ being defined in \eqref{eqBEvar}. This is the case for instance when the curvature is bounded from below. Then for all smooth $f$, $$|\nabla P_t f|(x)  \, \leq \, \mathbb E_x \left(e^{-\int_0^t \, \rho(X_s) \, ds} \, |\nabla f|(X_t)\right) \, .$$
\end{theorem}
\smallskip
Note that if $\rho$ is bounded from below by some constant (which may be negative) and $V$ is smooth enough, this theorem does apply. 

\subsection{Proof via coupling.\\ \\}\label{seccouple1}

We shall now follow the method used by two of us in \cite{CGsemin}, namely synchronous coupling. 
The argument below is very close to the one of \cite[Theorem 6]{Vey11}, that derived a strongly related Wasserstein distance estimate for diffusions on manifolds.

Since \eqref{eqsde} admits a strong solution, one can build with the \emph{same} Brownian motion a pair of solutions starting from $(x,y)$, denoted again by $(X_.^x,X_.^y)$. 

First recall that 
\begin{eqnarray*}
\langle \nabla V(z) -\nabla V(z') \, , \, z-z'\rangle &=& \int_0^1 \, \langle z-z' \, , \, \Hess V(\lambda z +(1-\lambda)z') \, (z-z')\rangle \, d\lambda \\ &\geq& \, \int_0^1  \, \rho(\lambda z +(1-\lambda)z') \, |z-z'|^2 \, d\lambda \, .
\end{eqnarray*}
Assuming again that the curvature is bounded from below, we thus have
\begin{equation}\label{eqcoup1}
e^{\int_0^t \, 2 \int_0^1  \, \rho(\lambda X_s^x +(1-\lambda)X_s^y) \, d\lambda \, ds} \, |X_t^x-X_t^y|^2 = |x-y|^2 \, - \, 2 \, A(t)
\end{equation}
with
\begin{eqnarray*}
A(t) &=& \int_0^t  \left(\langle \nabla V(X_s^x)-\nabla V(X_s^y) , X_s^x-X_s^y \rangle - \left(\int_0^1  \rho(\lambda X_s^x +(1-\lambda)X_s^y) \, d\lambda\right)|X_s^x-X_s^y|^2\right) \\ &  &  \, \quad \times \quad e^{\int_0^s \, 2 \int_0^1  \, \rho(\lambda X_u^x +(1-\lambda)X_u^y) \, d\lambda \, du} \, ds \, \\ & & \, \leq \, 0.
\end{eqnarray*}
Hence 
\begin{equation}\label{eqw1temp}
|X_t^x-X_t^y| \, \leq \, |x-y| \; \, e^{- \, \int_0^t \, \int_0^1  \, \rho(\lambda X_s^x +(1-\lambda)X_s^y) \, d\lambda \, ds} \, .
\end{equation}
Following \cite{CGsemin} p.5 we thus have, using the mean value theorem $$|P_tf(x)-P_tf(y)| \leq \mathbb E(|f(X_t^x)-f(X_t^y)|) \leq |x-y| \, \mathbb E\left(|\nabla f(z_t)| \,  e^{- \, \int_0^t \, \int_0^1  \, \rho(\lambda X_s^x +(1-\lambda)X_s^y) \, d\lambda \, ds}\right)$$ for some $z_t$ sandwiched by $X_t^x$ and $X_t^y$. It remains to divide by $|x-y|$, use that $X_s^y$ goes to $X_t^x$ almost surely as $y \to x$, that the curvature is bounded from below and Lebesgue's theorem to get a slightly different version of Theorem \ref{thmBEvarcommut}
\begin{theorem}\label{thmBEvarcommut2}
Assume that $V$ is $C^2$, that the curvature is bounded from below and that $\rho$ is continuous, $\rho$ being defined in \eqref{eqBEvar}. Then for all smooth $f$, $$|\nabla P_t f|(x)  \, \leq \, \mathbb E_x \left(e^{-\int_0^t \, \rho(X_s) \, ds} \, |\nabla f|(X_t)\right) \, .$$
\end{theorem}
\medskip

\section{A direct control of the $W_1$ distance.}\label{secWcouple}

Unlike what is done in \cite{CGsemin}, studying the evolution of the Wasserstein distance starting from \eqref{eqw1temp} is not immediate. The difference is that in the constant curvature case there is no need to interpolate between $X_t^x$ and $X_t^y$. 

\subsection{Reinforcing the Ricci bound. \\ \\ }\label{subsecreinf}
To bypass this issue, we shall reinforce inequality \eqref{eqBEvar} by assuming that there exists a function $\kappa : \R^n \longrightarrow \R$ such that for all $(x,y) \in \mathbb R^{2n}$
\begin{equation}\label{eqBEreinf}
\langle \nabla V(x)-\nabla V(y) \, , \, x-y\rangle \, \geq \, (\kappa(x)+\kappa(y)) \, |x-y|^2 \, .
\end{equation}
If this holds for some $\kappa$, taking limits $y \to x$ it is easily seen that $2\kappa(x) \leq \rho_m(x)$. We can then follow the previous proof and get
\begin{equation}\label{eqw1temp1}
|X_t^x-X_t^y| \, \leq \, |x-y| \; \, e^{- \, \int_0^t \, (\kappa(X_s^x) +\kappa(X_s^y) ) \, ds} \, .
\end{equation}
As a consequence,
\begin{equation}\label{eqw1xy}
W_1(P_t^*\delta_x \, , \, P_t^*\delta_y) \, \leq \, \mathbb E_x^{\frac 12}\left(e^{- \, \int_0^t \, 2 \kappa(X_s)  \, ds}\right) \, \mathbb E_y^{\frac 12}\left(e^{- \, \int_0^t \, 2 \kappa(X_s)  \, ds}\right) \, |x-y| \, .
\end{equation}
The advantage of this bound is that we do not need a sup-norm control anymore. Another approach, using a  related coarse, non-local version of curvature bounds was used in \cite{Vey11}. Of course, if we want to replace the Dirac masses by general measures, the situation is not as simple. Nevertheless we have for any coupling $\pi$ of $\nu$ and $\mu$,
\begin{eqnarray}\label{eqW1xmu}
W_1(P_t^*\nu\, , \, \mu) &\leq& \mathbb E\left(|X^\mu_0-X_0^\nu| \, e^{- \, \int_0^t \, \kappa(X^\mu_s)  \, ds} \, e^{- \, \int_0^t \, \kappa(X^\nu_s)  \, ds}\right) \nonumber \\ &\leq& \mathbb E_\pi^{\frac 1p} \left(|y-x|^p\right) \, \mathbb E_\mu^{\frac 1r}\left(e^{- \, \int_0^t \, r \, \kappa(X_s)  \, ds}\right) \, \mathbb E_\nu^{\frac 1q}\left(e^{- \, \int_0^t \, q \, \kappa(X_s)  \, ds}\right)
\end{eqnarray}
provided $\frac 1p + \frac 1q + \frac 1r = 1$. By choosing an appropriate coupling $\pi$, we then obtain the following: 
\begin{proposition}\label{propBEreinf}
Assume that $V$ is $C^2$ and that \eqref{eqBEreinf} is satisfied for some $\kappa$ which is bounded from below by $\kappa_0 \in \mathbb R$. Then for all $(p,q,r)$ with $\frac 1p + \frac 1q + \frac 1r = 1$, $$W_1(P_t^*\nu\, , \, \mu) \leq \, W_p(\nu,\mu)  \, \mathbb E_\mu^{\frac 1r}\left(e^{- \, \int_0^t \, r \, \kappa(X_s)  \, ds}\right) \, \mathbb E_\nu^{\frac 1q}\left(e^{- \, \int_0^t \, q \, \kappa(X_s)  \, ds}\right) \, .$$
\end{proposition}
Notice that since $2 \kappa \leq \rho_m$, the assumptions in the previous proposition imply that the curvature is bounded from below, hence that the process is conservative.

\begin{remark}
Since this proof is based on a coupling argument, we can also extend it to control $W_p$ distances with $p > 1$. 
\end{remark}

\begin{remark}\label{remBEreinf}
The set of $V$ satisfying \eqref{eqBEreinf} contains of course all uniformly convex functions (for which $\kappa$ is a positive constant), but also functions like $V(x)=g(|x|^2)$ for some convex function $g:\mathbb R^+ \to \mathbb R$ with $\kappa(x)=g'(|x|^2)$. Indeed for such functions 
\begin{eqnarray*}
\langle \nabla V(x) -\nabla V(y) \, , \, x-y\rangle &=& 2\left(g'(|x|^2)|x|^2 +g'(|y|^2)|y|^2 \, - \, (g'(|x|^2)+g'(|y|^2)) \, \langle x,y\rangle \right) \\ &\geq&  (g'(|x|^2)|x|^2 +g'(|y|^2)|y|^2) \, - \, 2(g'(|x|^2)+g'(|y|^2)) \, \langle x,y\rangle \\ &=& (g'(|x|^2)+g'(|y|^2))|x-y|^2
\end{eqnarray*}
since the difference equal to $$(g'(|x|^2)-g'(|y|^2))(|x|^2-|y|^2) \geq 0$$ thanks to the convexity of $g$.
\hfill $\diamondsuit$
\end{remark}
\begin{remark}\label{remmirrorcoup}
Actually, as explained in \cite[Remark 6]{CGsemin}, using \cite[Theorem 4.1 and below]{BGLharn}, one can deduce the gradient commutation property from the decay of the $W_1$ distance. This is done in \cite[Theorem 14]{CGsemin} to derive such a commutation under the assumption $$\langle \nabla V(x)-\nabla V(y) \, , \, x-y\rangle \, \geq \, \kappa(|x-y|) \, |x-y|^2 \, ,$$ replacing the synchronous coupling by the mirror coupling and assuming that $\kappa(u)$ goes to $\kappa_\infty >0$ as $u$ goes to $+ \infty$ ($\kappa(0)$ is not defined). This assumption is satisfied when $V$ is uniformly convex at infinity.
\hfill $\diamondsuit$
\end{remark}

\subsection{Some special cases. \\ \\}\label{subsec1d}

Assumption \eqref{eqBEreinf} is nevertheless much stronger than \eqref{eqBEvar}. It is thus natural to ask whether on can obtain similar results starting from \eqref{eqw1temp} $$|X_t^x-X_t^y| \, \leq \, |x-y| \; \, e^{- \, \int_0^t \, \int_0^1  \, \rho(\lambda X_s^x +(1-\lambda)X_s^y) \, d\lambda \, ds} \, .$$ The difficulty is to control $\rho(\lambda x +(1-\lambda)y)$ in terms of $\rho(x)$ and $\rho(y)$. In dimension one however we may use the following well known monotonicity property of the coupling we used: 
$$\textrm{if } \; x\leq y \; , \textrm{ then for all $s\in \mathbb R$, } \; X_s^x \leq X_s^y \; .$$ 
We may then state a first result: 
\begin{proposition}\label{1dlogconc}
Assume that $n=1$, $V$ is $C^2$ and that \eqref{eqBEvar} is satisfied for some $\rho$ which is bounded from below by $\rho_0 \geq 0$ (log-concave case). Assume in addition that $\rho$ is non-increasing on $(-\infty,a)$ and non-decreasing on $(a,+\infty)$ for some $a \in \mathbb R$. 

Then for all $(p,q,r)$ with $\frac 1p + \frac 1q + \frac 1r = 1$, $$W_1(P_t^*\nu \, , \, \mu) \leq \, W_p(\nu,\mu)  \, A_t$$ where 
\begin{eqnarray*}
A_t&=& \mathbb E_\mu^{\frac 1r}\left(e^{- \, \int_0^t \, r \, \mathbf 1_{X_s>a} \, \rho(X_s)  \, ds}\right) \, \mathbb E_\nu^{\frac 1q}\left(e^{- \, \int_0^t \, q \, \mathbf 1_{X_s<a} \, \rho(X_s)  \, ds}\right) + \\ &+& \, \mathbb E_\mu^{\frac 1r}\left(e^{- \, \int_0^t \, r \, \mathbf 1_{X_s<a} \, \rho(X_s)  \, ds}\right) \, \mathbb E_\nu^{\frac 1q}\left(e^{- \, \int_0^t \, q \, \mathbf 1_{X_s>a} \, \rho(X_s)  \, ds}\right).
\end{eqnarray*}
\end{proposition}
\begin{proof}
If $x \leq y$, and since $\rho(z) \geq 0$, it holds $$\rho(\lambda X_s^x +(1-\lambda)X_s^y) \, \geq \, \rho(X_s^x) \, \mathbf 1_{X_s^x>a} + \rho(X_s^y) \, \mathbf 1_{X_s^y<a} \, .$$ It follows 
\begin{eqnarray*}
W_1(P_t^*\nu\, , \, \mu) &\leq& \mathbb E\left(|X^\mu_0-X_0^\nu| \, \mathbf 1_{X_0^\mu \leq X_0^\nu} \; e^{- \, \int_0^t \, (\rho(X^\mu_s) \, \mathbf 1_{X_s^\mu>a} + \,\rho(X^\nu_s) \, \mathbf 1_{X_s^\nu<a}) \,  ds}\right) \, + \\ &+&  \mathbb E\left(|X^\mu_0-X_0^\nu| \, \mathbf 1_{X_0^\mu \geq X_0^\nu} \; e^{- \, \int_0^t \, (\rho(X^\mu_s) \, \mathbf 1_{X_s^\mu<a} + \,\rho(X^\nu_s) \, \mathbf 1_{X_s^\nu>a}) \,  ds}\right) 
\end{eqnarray*}
yielding the result.
\end{proof}
The assumptions on $\rho$ is of course satisfied in some cases, but if $a$ is not infinite it often implies that $V$ is uniformly convex at infinity. In order to extend the previous idea to the multi-dimensional setting we will introduce two different notions of lower bounds on the curvature, to control its behavior at infinity:
\begin{definition}\label{defradial}
Let $\rho$ satisfying \eqref{eqBEvar}. For $x \in \mathbb R^n$ we define $$\rho_{inf}(|x|) = \inf_{|y|\leq |x|} \, \rho(y) \quad \textrm{ and } \quad \rho_{sup}(|x|) = \inf_{|y|\geq |x|} \, \rho(y) \, .$$
\end{definition}
We have
\begin{proposition}\label{proplcamel}
Assume that $V$ is $C^2$ and that \eqref{eqBEvar} is satisfied for some $\rho$ which is bounded from below by $\rho_0 \geq 0$ (log-concave case). Then if $\rho_{inf}$ (resp. $\rho_{sup}$) is non-increasing (resp. non-decreasing) on $(a,+\infty)$ for some $a\geq 0$, for all $x$ and $y$ in $\mathbb R^n$, $$W_1(P_t^*\delta_x \, , \, P_t^*\delta_y) \, \leq \, |x-y| \; \mathbb E_x \left(e^{- \, \int_0^t \, \mathbf 1_{|X_s|>a+|x-y|} \, \rho_{inf}(|X_s|+|x-y|) \, ds}\right)$$ respectively $$W_1(P_t^*\delta_x \, , \, P_t^*\delta_y) \, \leq \, |x-y| \; \mathbb E_x \left(e^{- \, \int_0^t \, \mathbf 1_{|X_s|>a + |x-y|} \, \rho_{sup}(|X_s|-|x-y|) \, ds}\right) \, .$$ 
\end{proposition}
\begin{proof}
It is enough to remark that $$\rho(\lambda z +(1-\lambda)z') \geq \rho_{inf}(|z|+|z-z'|)$$ as soon as $|\lambda z +(1-\lambda)z'|\geq a$ and similarly for $\rho_{sup}$ and that $\rho \geq 0$ otherwise.
\end{proof}
\medskip

\section{Using WI deviation inequalities.}\label{secadditive}

\subsection{Transport-information inequalities. \\ \\}

As said in the introduction we shall extensively use the results of \cite{GLWY} on functional inequalities and concentration estimates for additive functionals. To this end we introduce some notations.
\smallskip

Let $c$ be a \emph{cost function} defined on $\mathbb R^n\otimes \mathbb R^n$, i.e $c$ is lower semicontinuous, non negative and satisfies $c(x,x)=0$ for all $x$. We shall also only consider cases where $c(x,y)=c(y,x)$. We may consider the \emph{transportation cost} $T_c$ defined by
\begin{equation}\label{eqtrans}
T_c(\nu,\mu) = \inf_{\{\pi, \pi\circ x^{-1}=\nu \, , \, \pi\circ y^{-1}=\mu\}} \, \int \, c(x,y) \, \pi(dx,dy) \, .
\end{equation}
When $c (x,y) = |x-y|^p$, we recover the usual $W_p$ distance, but other costs may also be of interest for our purpose here. The Kantorovich duality theorem states that 
\begin{equation}\label{eqtrans1}
T_c(\nu,\mu) = \sup_{(u,v)\in \Phi_c} \, \nu(u) \, - \, \mu(v) \, ,
\end{equation}
where $$\Phi_c := \{(u,v) \textrm{ Borel and bounded s.t. } u(x)-v(y)\leq c(x,y) \, , \forall (x,y)\} \, .$$ Also recall the definition of the Fischer information: for $\nu=f \mu$ with $f>0$ and locally Lipschitz
\begin{equation}\label{eqfisch}
I(\nu|\mu) := \frac 14 \; \mu\left(\frac{|\nabla f|^2}{f}\right) \, = \, \mu\left(|\nabla (\sqrt f)|^2\right) \, .
\end{equation}
If $\nu$ does not satisfy the previous assumptions then, $I(\nu|\mu)=+\infty$.

The following is a consequence of Theorem 2.2 in \cite{GLWY}, connecting functional inequalities and deviation estimates of additive functionals along drift-diffusion S.D.E.

\begin{theorem}\label{thm1cgwy}
Let $\alpha:[0,+\infty) \to [0,+\infty)$ be a continuous increasing function with $\alpha(0)=0$. Then the following statements are equivalent.
\begin{enumerate}
\item[(1)] \quad The following transport-information $WI(\alpha,c)$ inequality is satisfied $$\alpha(T_c(\nu,\mu)) \, \leq \, I(\nu|\mu) \; , \quad \forall \nu \, .$$
\item[(2)] \quad For any initial measure $\beta \ll \mu$ such that $d\beta/d\mu \in \mathbb L^2(\mu)$, all $u,v$ such that $(u,v) \in \Phi_c$, all $R,t>0$, $$\mathbb P_\beta \left(\frac 1t \, \int_0^t u(X_s) \, ds \, \geq \, \mu(v) \, + \, R \, \right) \; \leq \; ||d\beta/d\mu||_{\mathbb L^2(\mu)} \; e^{ \, - \, t \, \alpha(R)} \, .$$
\item[(3)] \quad Defining $$P_t^{\lambda u}f(x) \, = \, \mathbb E_x\left[f(X_t) \, e^{\int_0^t \, \lambda \, u(X_s) \, ds}\right] \, , $$ then for all $u,v$ such that $(u,v) \in \Phi_c$ and all $t>0$ it holds $$\parallel P_t^{\lambda u}f\parallel_{\mathbb L^2(\mu)} \, \leq \, e^{t \, (\lambda \, \mu(v) \, + \, \alpha^*(\lambda))} \; \parallel f \parallel_{\mathbb L^2(\mu)} \, ,$$ where $\alpha^*$ is the Legendre transform of $\alpha \mathbf 1_{[0, \infty)}$, restricted to $[0, \infty)$. 
\smallskip

\noindent In particular, for any initial measure $\beta \ll \mu$ such that $d\beta/d\mu \in \mathbb L^2(\mu)$,
$$\mathbb{E}_{\beta}\left[\exp\left(\int_0^t{ \lambda \; u(X_s)ds}\right)\right] \leq ||d\beta/d\mu||_{\mathbb L^2(\mu)} \, \exp[t \, (\lambda \mu(v) + \alpha^*(\lambda))] \, .$$
\end{enumerate}
\end{theorem}

When curvature is bounded from below by a positive constant, a quadratic transport-information inequality holds, with the $L^2$ Wasserstein distance. Concentration estimates for additive functionals can also be directly derived from positive curvature \cite{JO10}. However, positive curvature is not a necessary condition, and we can also use non-quadratic inequalities. {There are also direct approaches starting from Poincar\'e inequality such as \cite{CGesaim} with bounded curvature or \cite{AJLR} only requiring  curvature bounded from below, but whose constants may not lead to a strict improvement of our results.}

\begin{definition}\label{deflip}
Let $c$ be a cost function and $u$ a measurable function. We define $$\parallel u\parallel_c := \, \sup_{x \neq y} \, \frac{|u(x)-u(y)|}{c(x,y)} \, \in [0, \infty].$$ 
\end{definition}

\begin{remark}\label{remlipc} 
Notice that a simple monotone convergence argument allows us to extend Theorem \ref{thm1cgwy} to non necessarily bounded functions $u$ such that $\parallel u\parallel_c =1$. This will be used in the sequel without further mention. \hfill $\diamondsuit$
\end{remark}

When $u$ is bounded from below we can derive another bound
\begin{corollary}\label{propcontroleexp}
Let $\alpha:[0,+\infty) \to [0,+\infty)$ be a continuous increasing function with $\alpha(0)=0$. Let $u$ be a measurable function bounded from below by $u_0 \in \mathbb R$ and such that $\mu(u)>0$. Then for all initial measure $\beta \ll \mu$ such that $d\beta/d\mu \in \mathbb L^2(\mu)$, all $\lambda$ and $t$ strictly positive, all $1>\varepsilon >0$
$$\mathbb{E}_{\beta}\left[\exp\left(- \, \int_0^t{ \lambda \, u(X_s)ds}\right)\right] \leq $$ $$\left(1 + ||d\beta/d\mu||_{\mathbb L^2(\mu)}\right) \, \max\left(e^{- \, \lambda \, \varepsilon \, \mu(u) \, t} \; , \; e^{- \, t \, (\lambda \, u_0  \, + \, \alpha((1-\varepsilon) \mu(u)/||u||_{c}))}\right) \, .$$  
\end{corollary}
\begin{proof} Introduce the set $$A = \left\{- \, \int_0^t \, \lambda \, (u -\mu(u))(X_s) \, ds \, \geq  \, Rt \right\} \, .$$ 
We thus have,  using Theorem \ref{thm1cgwy},
\begin{eqnarray}\label{eqw1exp}
E_\beta\left(e^{- \, \int_0^t \, \lambda \, u(X_s)  \, ds}\right)&\leq& e^{- \, \lambda \, u_0 \, t} \, \mathbb P_\beta(A) \,  + \, e^{- \lambda \, \mu(u) \, t} \, E_\beta\left(e^{- \, \int_0^t \, \lambda \, (u -\mu(u))(X_s)  \, ds} \, \mathbf 1_{A^c}\right) \nonumber \\ &\leq& ||d\beta/d\mu||_{\mathbb L^2(\mu)} \; e^{t \, (- \, \lambda \, u_0  \, - \, \alpha(R/\lambda||u||_{c}))} \, + \, e^{- \, (\lambda \, \mu(u) \, - \, R) \, t} \, .
\end{eqnarray}
 Choosing $R=\lambda \, \mu(u) \, (1-\varepsilon)$ gives the result.
\end{proof}
\medskip

\begin{remark} It is also possible to obtain exponential convergence of additive functionals via Lyapunov-type conditions, see for example \cite{KM05, FRS18}, but the constants obtained there are not so explicit, and typically depend on the dimension, so they are not so suitable for our purpose here. 
\end{remark}

\subsection{Rates of convergence in $W_1$ distance. \\ \\}
We start with an immediate application of what precedes: 

\begin{theorem}\label{propkappalip}
Assume that $V$ is $C^2$, that \eqref{eqBEreinf} is satisfied with some function $\kappa$ which satisfies $\parallel \kappa\parallel_c <+\infty$, bounded from below by some constant $\kappa_0 \in \mathbb R$ and such that $\mu(\kappa)>0$. Assume that $\mu$ satisfies a $WI(\alpha,c)$ inequality for some function $\alpha$. Then, for all $\nu$ such that $d\nu/d\mu \in \mathbb L^2(\mu)$, all $t>0$, $$W_1(P_t^*\nu\, , \, \mu) \, \leq \, C(\nu) \, e^{- \, \theta \, t} \, ,$$ in the following cases:
\begin{enumerate}
\item[(1)] \quad If $\alpha^*$ satisfies 
\begin{equation}\label{eqgoodalphastar}
\mu(\kappa) > \frac{\alpha^*(2 \, \parallel \kappa \parallel_c)}{2} \, ,
\end{equation}
then for any $r>2$ such that $\mu(\kappa) > \frac{\alpha^*(r \, ||\kappa||_{c})}{r}$, and $p=r/(r-2)$, we may choose $$ C(\nu) = ||d\nu/d\mu||_{\mathbb L^2(\mu)}^{\frac{1}{r}} \, W_p(\nu,\mu)$$ and $$\theta = 2 \, (\mu(\kappa) \, - \, (\alpha^*(r \, \parallel \kappa \parallel_c)/r )) \, .$$ 
\item[(2)] \quad If $\alpha$ satisfies 
\begin{equation}\label{eqgoodalpha1}
2 \, \kappa_0+\alpha(\mu(\kappa)/\parallel \kappa \parallel_c) > \, 0 \, ,
\end{equation}
then for any $r>2$, $1>\varepsilon>0$ such that $r \, \kappa_0+\alpha((1-\varepsilon)\mu(\kappa)/\parallel \kappa \parallel_c) >0$, and $p=r/(r-2)$, we may choose $$C(\nu) \, = \, (1+||d\nu/d\mu||_{\mathbb L^2(\mu)})^{\frac 1r} \, W_p(\nu,\mu) \, ,$$ and $$\theta \, = \, 2 \,  \min \left(\varepsilon \, \mu(\kappa) \, , \, \kappa_0 + \frac 1r \, \alpha((1-\varepsilon) \mu(\kappa)/||\kappa||_{c})\right) \, .$$
\end{enumerate}
\end{theorem}
\begin{proof}
According to Proposition \ref{propBEreinf}, it is enough to control $\mathbb E_\nu^{\frac 1r}\left(e^{- \, \int_0^t \, r \, \kappa(X_s)  \, ds}\right)$ for some well chosen $r\geq 2$. 

Part (1) is an immediate consequence of (3) in Theorem \ref{thm1cgwy} applied with $u=-\kappa/\parallel \kappa \parallel_c$ and $\lambda= r \, \parallel \kappa \parallel_c$.

Part (2) follows from Corollary \ref{propcontroleexp}. When $r \, \kappa_0+\alpha((1-\varepsilon)\mu(\kappa)/\parallel \kappa \parallel_c) >0$, both exponential terms go to $0$, so that the exponential decay rate is given by $$\min \left(\varepsilon \, \mu(\kappa) \, , \, \kappa_0 + \frac 1r \, \alpha((1-\varepsilon) \mu(\kappa)/||\kappa||_{c})\right) \, .$$
\end{proof}

We thus have a balance between the rate of exponential convergence and the initial control i.e. $r$ and $p$ or $\varepsilon$. Note that since we want some regularity for $\kappa$, it may be that the one to use is not the one that optimizes \eqref{eqBEreinf}, but some smoother minorant. 

One can notice that in case (2) the rate is always less than or equal to $\varepsilon \, \mu(\kappa)$ and that it can be achieved by choosing $$r(\varepsilon)=\frac{\alpha((1-\varepsilon) \mu(\kappa)/||\kappa||_{c})}{\varepsilon \mu(\kappa) - \kappa_0 } \, ,$$ provided the latter is larger than $2$ and $\varepsilon \mu(\kappa) - \kappa_0 >0$. We shall come back to this issue in Section \ref{secapplis}.
\medskip

In order to extend the result to more general initial measures we will use in addition the following regularization result
\begin{proposition}\label{propregul}
Assume that $V$ is $C^2$ and that the process is conservative. Then for all $x \in \mathbb R^n$ and $t>0$, $P^*_t \delta_x(dy) \ll dy$ and its density $p(t,x,y)$ is continuous (in $y$).

We will denote $r(t,x,y)=p(t,x,y) \, e^{V(y)}$ the density of $P^*_t \delta_x$ w.r.t. $\mu$. Notice that $r(t,x,y)=r(t,y,x)$.
\end{proposition}
This result is certainly well known in P.D.E. theory. The assumptions we are using are those of \cite[Proposition 4.2]{MNS}, where the proof is based on Malliavin calculus.

Then, as argued in \cite{CGZ}, symmetry combined with the Chapman-Kolmogorov relation and the continuity of $r$ yield for all $t>0$, $$\int \, r^2(t,x,y) \, \mu(dy) = \int \, r(t,x,y) \, r(t,y,x) \, \mu(dy) =  r(2t,x,x) \, < +\infty \, .$$ If $\nu$ is any probability measure, $P_t^*\nu$ is thus absolutely continuous w.r.t. $\mu$ with density $r(t,\nu,y)=\int \, r(t,x,y) \, \nu(dx)$ which belongs to $\mathbb L^2(\mu)$ as soon as $\int \, r(2t,x,x) \, \nu(dx) < +\infty$. 
\smallskip

It follows that for all $\eta >0$, 
\begin{eqnarray*}
\mathbb E_\nu\left(e^{- \, r \, \int_0^t  \, \kappa(X_s)  \, ds}\right)&=& \mathbb E_\nu\left(e^{- \, r \, \int_0^\eta  \, \kappa(X_s)  \, ds} \, \mathbb E_{X_\eta}\left(e^{- \, r \, \int_0^{t-\eta}  \, \kappa(X'_s)  \, ds}\right)\right) \\ &\leq& e^{- r \, \kappa_0 \, \eta} \, \mathbb E_{P^*_\eta \nu}\left(e^{- \, r \, \int_0^{t-\eta}  \, \kappa(X_s)  \, ds}\right)
\end{eqnarray*}
Hence as before we easily get
\begin{eqnarray*}
\mathbb E_\nu\left(e^{- \, \int_0^t  \, \kappa(X_s)  \, ds}\right) &\leq& \mathbb E_\nu^{\frac 1r} \left(e^{- \, r \, \int_0^t  \, \kappa(X_s)  \, ds}\right)\\ &\leq& e^{- \, \kappa_0 \, \eta} \, \mathbb E_{P^*_\eta \nu}^{\frac 1r}\left(e^{- \, r \, \int_0^{t-\eta}  \, \kappa(X_s)  \, ds}\right)
\end{eqnarray*}
and we may argue as in the previous proof (using that the $\mathbb L^2$ norm of a probability density is larger than $1$) to obtain
\begin{corollary}\label{corkappalip}
The conclusions of Theorem \ref{propkappalip}, are true for any probability measure $\nu$, any $t>\eta>0$, with the same $\theta $ but $$C(\nu) \, = \, 2^{\frac 1r} \, \left(\int \, r(2\eta,x,x) \, \nu(dx)\right)^{\frac{1}{2r}} \; e^{(\theta - \kappa_0) \, \eta} \; W_p(\nu,\mu) \, ,$$ where $r(\eta,.,.)$ is defined in proposition \ref{propregul}.
\end{corollary}
\medskip

In order to consider the more general curvature assumption \eqref{eqBEvar} we would have to get uniform bounds in $x$ when the initial measure is $\delta_x$ (see \eqref{eqw1-1}).  We can follow the same path as before provided $\rho$ satisfies the same assumptions as $\kappa$. Notice that this time we only need $r \geq 1$.
\begin{theorem}\label{proprholip}
Assume that $V$ is $C^2$, that \eqref{eqBEvar} is satisfied with some function $\rho$ which satisfies $\parallel \rho\parallel_c <+\infty$, is bounded from below by some constant $\rho_0 \in \mathbb R$ and such that $\mu(\rho)>0$. Assume moreover that $\rho$ is continuous or $V$ is $C^3$ and $\rho$ is locally bounded. Finally, assume that $\mu$ satisfies a $WI(\alpha,c)$ inequality for some function $\alpha$. Then, for all $\nu$, all $\eta >0$ and all $t>\eta$, $$W_1(P_t^*\nu\, , \, \mu) \, \leq \, C(\nu) \, e^{- \, \theta \, t} \; W_1(\nu,\mu) ,$$ with $$ C(\nu) =  2^{\frac 1r} \, \left(\sup_x \, r(2\eta,x,x)\right)^{\frac{1}{2r}} \; e^{(\theta - \rho_0) \, \eta}  \, ,$$ in the following cases:
\begin{enumerate}
\item[(1)] \quad If $\alpha^*$ satisfies 
\begin{equation}\label{eqgoodalphastarbis}
\mu(\rho) > \alpha^*(\parallel \rho \parallel_c) \, ,
\end{equation}
for all $r\geq 1$ such that $\mu(\rho) > \frac{\alpha^*(r \, ||\rho||_{c})}{r}$, we may choose $$\theta =  (\mu(\rho) \, - \, (\alpha^*(r \, \parallel \rho \parallel_c)/r )) \, .$$ 
\item[(2)] \quad If $\alpha$ satisfies 
\begin{equation}\label{eqgoodalpha1bis}
\rho_0+\alpha(\mu(\rho)/\parallel \rho \parallel_c) > \, 0 \, ,
\end{equation}
for all $r \geq 1$, $1>\varepsilon>0$ such that $r \, \rho_0+\alpha((1-\varepsilon)\mu(\rho)/\parallel \rho \parallel_c) >0$, we may choose $$\theta \, =  \,  \min \left(\varepsilon \, \mu(\rho) \, , \, \rho_0 + \frac 1r \, \alpha((1-\varepsilon) \mu(\rho)/||\rho||_{c})\right) \, .$$
\end{enumerate}
\end{theorem}
\begin{remark}\label{remsupr}
One can think that the assumption used here on the density $r(2\eta,.,.)$ is too strong. Consider for example the Ornstein-Uhlenbeck corresponding to the standard gaussian measure. Then $$r(t,x,x) = c(t) \, e^{- \, \frac 12 \, |x|^2 \, \left(\frac{(1-e^{-t})^2}{1-e^{-2t}} \, - \, 1\right)} = c(t) \, e^{|x|^2 \, e^{-t} (1-e^{-t})} \, ,$$ so that it is not bounded. However if the semi-group is \emph{ultra-bounded}, i.e. maps continuously $\mathbb L^1(\mu)$ in $\mathbb L^\infty(\mu)$ for all $t>0$, then, since $r(\eta,x,.)$ is a density of probability w.r.t. $\mu$, $r(2\eta,x,.)$ is bounded, with a sup norm only depending on $\eta$, and not of $x$. Actually we only need the $(\mathbb L^1,\mathbb L^\infty)$ continuity for a single time $\eta$. \hfill $\diamondsuit$
\end{remark}
\smallskip

\subsection{Entropic pre-factor. \\ \\}

When curvature is nonnegative, we can get a variant of Theorem \ref{propkappalip} with an entropic pre-factor instead of an $L^2$ norm, but at the cost of weakening the rate of convergence. We recall the definition of the relative entropy functional
$$H(\nu|\mu) := \int{g \log g d\mu}; \hspace{3mm} \text{if} \hspace{1mm} d\nu = g d\mu$$
and takes value $+\infty$ if $\nu$ is not absolutely continuous with respect to $\mu$.  

\begin{theorem}
Assume that the transport-information inequality
$$\alpha\left(T_c(\nu, \mu)\right) \leq I(\nu|\mu)$$
holds, and that $||\kappa||_c < \infty$. Assume moreover that $\kappa \geq \kappa_0$ for some $\kappa_0 \in \mathbb R$ and that there exists $q > 1$ such that
$$(q||\kappa||_c)^{-1}\alpha^*(q||\kappa||_c) \leq \mu(\kappa) + \kappa_0 \, .$$ 
Then for any $r, p$ with $\frac 1p + \frac 1q + \frac 1q = 1$ and $\varepsilon \in (0,1)$ such that $(1-\varepsilon) \mu(\kappa) > (q||\kappa||_c)^{-1}\alpha^*(q||\kappa||_c)$, we have
$$W_1(P_t^*\nu, \mu) \leq  W_p(\nu, \mu) \, (H_1(t) + H_2(t))$$ with $$H_1(t)= \left(\frac{\log 2 + H(\nu|\mu)}{t\alpha((1-\varepsilon)/||\kappa||_c)}\right)^{1/r}e^{- \, (\mu(\kappa) + \kappa_0 - (q||\kappa||_c)^{-1}\alpha^*(q||\kappa||_c))t} \, ,$$ and $$H_2(t) = e^{- \, ((1+\varepsilon)\mu(\kappa) - (q||\kappa||_c)^{-1}\alpha^*(q||\kappa||_c))t} \, .$$
\end{theorem}

Since $H(\nu|\mu) \leq \log(1 + ||\rho||_{L^2(\mu)})$, this entropic pre-factor is often much smaller than the $L^2$ pre-factor of Theorem \ref{propkappalip}. 

\begin{proof}
We shall use the classical entropy inequality
\begin{equation} \label{ineq_ent}
P(A) \leq \frac{\log 2 + H(P|Q)}{\log(1 + 1/Q(A))}
\end{equation}
for any two probability measures $P$ and $Q$, and event $A$. It can be derived as a consequence of the dual formulation
$$H(P|Q) = \sup_f P(f) - \log Q(e^f)$$
by taking $f = \lambda \mathbf 1_A$ and then optimizing in $\lambda$. 

We wish to apply \eqref{ineq_ent} with $P$ the distribution of a trajectory of the diffusion process with initial data distributed according to $\nu$, and $Q$ the law of a trajectory of the same process, but started from the equilibrium distribution $\mu$. It is known, for example as a consequence of the Girsanov formula \cite{Leo12}, that in this case
$$H(P|Q) = H(\nu|\mu).$$
If we take $A$ the event ${\int_0^t{\kappa(X_s)ds} \geq \mathbb{E}_{\mu}(\kappa) + r}$, we get $$\mathbb P(A) \leq (\log 2 + H(\nu|\mu))/(t\alpha(r)).$$
The proof of the first part then proceeds as in the proof of Theorem \ref{propkappalip}. Combining the transport-information inequality and the above bound, the counterpart to \eqref{eqw1exp} for $r, R$ and $A$ as in the proof of Theorem \ref{propkappalip}, is
\begin{eqnarray}\label{eqw1exp_ent}
\mathbb E_\nu\left(e^{- \, \int_0^t \, r \, \kappa(X_s)  \, ds}\right)&\leq& e^{- \, r \, \kappa_0 \, t} \, \mathbb P_\nu(A) \,  + \, e^{-r \, \mu(\kappa) \, t} \, \mathbb E_\nu\left(e^{- \, \int_0^t \, r \, (\kappa -\mu(\kappa))(X_s)  \, ds} \, \mathbf 1_{A^c}\right) \nonumber \\ &\leq& \frac{\log 2 + H(\nu|\mu)}{t\alpha(R/(r||\kappa||_c))} \; e^{ \, - \, r \, \kappa_0 t} \, + \, e^{- \, (r \, \mu(\kappa) \, - \, R) \, t} \, .
\end{eqnarray}
We then take as before, $R = (1-\varepsilon) r\mu(\kappa)$ and get
$$
\mathbb E_\nu\left(e^{- \, \int_0^t \, r \, \kappa(X_s)  \, ds}\right)^{1/r} \leq \left(\frac{\log 2 + H(\nu|\mu)}{t\alpha((1-\varepsilon)/||\kappa||_c)}\right)^{1/r}e^{-\kappa_0 t} +e^{- \, \varepsilon \, \mu(\kappa) \, t} \, .$$
Moreover, 
$$\mathbb{E}_{\mu}\left[\exp\left(\int_0^t{-q\kappa(X_s)ds}\right)\right]^{1/q} \leq  \exp(-t\mu(\kappa) + t(q||\kappa||_c)^{-1}\alpha^*(q||\kappa||_c)).$$

Using Proposition \ref{propBEreinf} then concludes the proof. 
\end{proof}

We could also straightforwardly derive results similar to Corollary \ref{corkappalip} and Theorem \ref{proprholip}, but for the sake of brevity we do not do so here.
\medskip

\subsection{Bounds on the Poincar\'e constant. \\ \\}

Recall that according to the discussion in the introduction, thanks to \cite[Lemma 2.12]{CGZ}, $C_P(\mu) \leq \frac 2c$ as soon as $$\mathbb E_\mu\left(e^{-2\int_0^t \, \rho(X_s) ds}\right) \leq C \, e^{-ct} \, .$$

For $r>2$ we may use
\begin{eqnarray*}
\mathbb E_\mu\left(e^{- \, \int_0^t  \, 2 \, \rho(X_s)  \, ds}\right) \, &\leq& \mathbb E_\mu^{\frac 2r}\left(e^{- \, \int_0^t  \, r \, \rho(X_s)  \, ds}\right) \, ,
\end{eqnarray*}
together with Theorem \ref{thm1cgwy} or Corollary \ref{propcontroleexp}. This yields
\begin{theorem}\label{propPoinclip}
Assume that the assumptions of Theorem \ref{proprholip} are fulfilled. 
\begin{enumerate}
\item[(1)] \quad if $\alpha^*$ satisfies 
\begin{equation}\label{eqgoodalphastarter}
\mu(\rho) > \alpha^*(2\parallel \rho \parallel_c)/2 \, ,
\end{equation}
then $$C_P(\mu) \, \leq \, \frac{1}{\mu(\rho)-\alpha^*(2 \parallel \rho\parallel_c)/2} \, .$$
\item[(2)] \quad If $\alpha$ satisfies 
\begin{equation}\label{eqgoodalpha1ter}
2\rho_0+\alpha(\mu(\rho)/\parallel \rho \parallel_c) > \, 0 \, ,
\end{equation}
then for all $r>2$, $1>\varepsilon>0$ such that $r \, \rho_0+\alpha((1-\varepsilon)\mu(\rho)/\parallel \rho \parallel_c) >0$, $$C_P(\mu) \, \leq \, \frac{1}{\min (\varepsilon \, \mu(\rho) \, , \, \rho_0 + (\alpha((1-\varepsilon)\mu(\rho)/\parallel \rho\parallel_c)/r))} \, .$$
If $c$ is continuous, the optimal $\varepsilon$ satisfies $$\alpha((1-\varepsilon)\mu(\rho)/\parallel \rho\parallel_c)=2(\varepsilon \mu(\rho) - \rho_0)$$ provided the right hand side is non-negative and in this case $C_P(\mu) \leq 1/(\varepsilon \, \mu(\rho))$. 
\end{enumerate}
\end{theorem}
For the last assertion, first remark that for a given $\varepsilon$ the optimal $r$ is such that both terms in the $min$ are equal, i.e. $$r(\varepsilon) = \frac{\alpha((1-\varepsilon)\mu(\rho)/\parallel \rho\parallel_c)}{\varepsilon \mu(\rho) - \rho_0}$$ provided $r(\varepsilon) \geq 2$. Then remark that $\varepsilon \mapsto h(\varepsilon)= \alpha((1-\varepsilon) \mu(\rho)/||\rho||_{c})$ is non increasing, $\varepsilon \mapsto g(\varepsilon) = 2(\varepsilon \mu(\rho)-\rho_0)$ is non decreasing, they are both continuous and $h(0)>g(0)$ while $h(1)<g(1)$ yielding the optimal $\varepsilon$. 
\smallskip

\begin{remark}\label{remnosmooth}
If $V$ is not smooth enough, the previous result is still true in the following situation: there exists a family $V_\eta$ of smooth potentials (for instance $V*\gamma_\eta$ where $\gamma_\eta$ is a centered gaussian kernel with covariance matrix $\eta Id$) and a bounded $\rho$ such that $\rho_\eta \geq \rho$ for all $\eta$ and the measure $e^{-V_\eta} dx$ converges weakly to $\mu$ as $\eta \to 0$. \hfill $\diamondsuit$
\end{remark}
\smallskip

\begin{remark}\label{remcost}
We will now discuss some possible choices for the cost $c$. Notice the transport-information $WI(\alpha,c)$ is only important for the single function $\rho$ (or $\kappa$). A natural choice will thus be
\begin{equation}\label{eqchoicedist}
c(x,y) \, = \, |\rho(x) \, - \, \rho(y)| \, ,
\end{equation}
so that $||\rho||_{c}=1$. The difficulty will be to find tractable conditions for $WI(\alpha,c)$, in particular because $c$ is not necessarily a distance.  
\smallskip

The most classical choice is $c(x,y)=|x-y|$ so that $\parallel u\parallel_c = \parallel u\parallel_{Lip} =\parallel |\nabla u|\parallel_\infty$. $WI(\alpha,c)$ then becomes the following $W_1I(\alpha)$ inequality  
\begin{equation}\label{eqW1Ialpha}
\alpha(W_1(\nu,\mu))  \, \leq \, I(\nu|\mu) \; , \quad \forall \nu \, .
\end{equation}

Of course, in many situations one may replace the \emph{natural} $\rho$ or $\kappa$ by some smaller function. For instance, if we consider $V(x)=|x|^4$ as in remark \ref{remBEreinf}, we may replace the natural $\kappa(x)=2|x|^2$ by $\kappa(x)=2\min(|x|^2,|x|)$ which is (globally) Lipschitz, and thus apply the previous results provided $W_1I(\alpha)$ is satisfied. 
\smallskip

If we choose the Hamming distance $c(x,y)=\mathbf 1_{x \neq y}$, $T_c(\nu,\mu)= \frac 12 \, \parallel \nu - \mu\parallel_{TV}$ and $\parallel u \parallel_c \, = \, \Osc (u) = \sup_{x,y} \, |u(x)-u(y)|$. 
\hfill $\diamondsuit$
\end{remark}
\medskip

\begin{remark}{\bf{About homogeneity.}}

It is well known that the Poincar\'e inequality is $2$-homogeneous w.r.t. dilations. This means the following: for $\lambda >0$, define $\mu_\lambda(f)=\int f(\lambda x) \mu(dx)$ so that 
\begin{eqnarray*}
\mu_\lambda(dx) &=& \frac 1\lambda \, e^{-V(x/\lambda)} \, dx \, := \, e^{-V_\lambda(x)} \, dx \, ,\\
V_\lambda(x) &=& \ln(\lambda) + V(x/\lambda) \, , \quad \nabla V_\lambda(x)=\frac 1\lambda \, \nabla V(x/\lambda) \, , \quad \rho_\lambda(x) = \frac{1}{\lambda^2} \, \rho(x/\lambda) \, , \\ \mu_\lambda(\rho_\lambda) &=& \int \rho_\lambda(\lambda x) \mu(dx) \, = \, \frac{1}{\lambda^2} \, \mu(\rho) \, ,  \quad \rho_{\lambda,0}=\frac{1}{\lambda^2} \, \rho_0 \, \\ \parallel \rho_\lambda\parallel_{Lip}&=& \frac{1}{\lambda^3} \, \parallel \rho \parallel_{Lip} \, , \quad \Osc(\rho_\lambda) = \frac{1}{\lambda^2} \, \Osc(\rho) \, , 
\end{eqnarray*}
and finally $$C_P(\mu_\lambda) = \lambda^2 \, C_P(\mu) \, .$$ 
\medskip

Of course the key point here is that the dilation only concerns the invariant measure and not the dynamics. In particular there is no immediate correspondence between the natural dynamics corresponding to $\mu$ and its dilation, so that we cannot expect to get a direct comparison of the rate of convergence in $W_1$ distance for both. 
\smallskip

Now consider $T_c(\mu_\lambda,\nu)$. If we assume that $c$ is homogeneous of degree $\eta$, i.e. $$c(\lambda x,\lambda y)=\lambda^\eta \, c(x,y) \, $$ we have $$T_c(\mu_\lambda,\nu) = \lambda^\eta \, T_c(\mu,\nu_{1/\lambda}) \, .$$ It is easy to check that $$I(\nu_{1/\lambda}|\mu) = \lambda^2 \, I(\nu|\mu_\lambda) \, ,$$ so that if $\mu$ satisfies a $WI(\alpha,c)$ inequality, 
\begin{eqnarray*}
T_c(\nu,\mu_\lambda) &=& \lambda^\eta \; T_c(\nu_{1/\lambda},\mu) \leq  \lambda^\eta \, \alpha^{-1}(I(\nu_{1/\lambda}|\mu)) \\ &\leq& \lambda^\eta \, \alpha^{-1}(\lambda^2 \, I(\nu|\mu_\lambda)) \, ,
\end{eqnarray*}
and $\mu_\lambda$ satisfies a $WI(\alpha_\lambda,c)$ inequality with $$\alpha_\lambda^{-1}(s)=\lambda^{\eta} \, \alpha^{-1}(\lambda^2 \, s) \, .$$ For simplicity we will only consider the case $\alpha(s)= c s^m$ for some $m>0$, i.e. $$\alpha_\lambda(s)=c \, \lambda^{-(m\eta+2)} \, s^m \, .$$ 
If we consider both cases, $c(x,y)=|x-y|$, yielding $\eta=1$, or $c(x,y)=\mathbf 1_{x\neq y}$ yielding $\eta=0$, it is easily seen that our Theorem \ref{propPoinclip} is homogeneous w.r.t. to dilations. \hfill $\diamondsuit$
\end{remark}
\medskip

\begin{remark}\label{remgamma2integre}
The criterion $\mu(\rho) > 0$ can be called positive \emph{average curvature}. Recall that in the symmetric situation we are looking at, a Poincar\'e inequality is equivalent to an integral $\Gamma_2$ criterion (see e.g \cite{Ane} proposition 5.5.4), namely 
\begin{proposition}\label{propintegregamma2}
The following statements are equivalent
\begin{itemize}
\item For all smooth $f$,  $$\mu(\parallel \Hess f\parallel^2_2) + \mu(\langle \nabla f \, , \, \Hess V \, \nabla f\rangle) \, \geq \, C \, \mu(|\nabla f|^2) \, ,$$
\item The Poincar\'e constant satisfies $$C_P(\mu) \, \leq \, 1/C \, ,$$
\end{itemize}
where $\parallel \Hess f(x)\parallel_2$ denotes the Hilbert-Schmidt norm of $\Hess f(x)$.\end{proposition}
In particular if  \eqref{eqBEvar} is satisfied, $C_P(\mu) \leq 1/C$ as soon as $\mu(\rho |\nabla f|^2) \, \geq \, C \, \mu(|\nabla f|^2)$ for all smooth $f$.  If we may choose a positive constant for $\rho$ we recover the Bakry-Emery criterion, but if $\rho(x) < - \varepsilon$ on a small ball, one can build some $f$ such that the previous inequality is not satisfied. Hence our Theorem \ref{propPoinclip} is of a different nature.
\hfill $\diamondsuit$
\end{remark}

\medskip

\section{Examples and Applications.}\label{secapplis}

\subsection{About $W_1I$ inequalities.\\ \\}\label{subsecenough}

Since the results of the previous section are based on $W_1I$ inequalities, it is important to have some sufficient conditions for these inequalities to be satisfied. Actually not so much is known and we recall below the main examples:
\begin{proposition}\label{propcompare}
We have the following properties
\begin{enumerate}
\item[(1)] \; (see \cite[Theorem 3.1]{GLWY}) \quad Assume that $\mu$ satisfies a Poincar\'e inequality  with optimal constant $C_P(\mu)$, then $\mu$ satisfies a $W_1I$ inequality $$\parallel \nu - \mu\parallel_{TV}^2 \, \leq \, 4 \, C_P(\mu) \, I(\nu|\mu) \quad \forall \, \nu \, ;$$ i.e. a $WI(\alpha,c)$ inequality with $\alpha(s)= (1/C_P(\mu)) \, s^2$ and $c(x,y)= \mathbf 1_{x \neq y}$.
\item[(2)] \; (see \cite{OV}) \quad Assume that  $\mu$ satisfies a logarithmic Sobolev inequality  with optimal constant $C_{LS}(\mu)$. Then $\mu$ satisfies a $W_2I$ inequality $$W_2^2(\nu \, , \, \mu) \, \leq \, C_{LS}^2(\mu) \, I(\nu|\mu) \quad \forall \, \nu \, .$$ Consequently, since $W_1 \leq W_2$,  it satisfies a $WI(\alpha,c)$ inequality with $\alpha(s)= (1/C_{LS}^2(\mu)) \, s^2$ and $c(x,y)=|x-y|$.
\end{enumerate}
Of course similar inequalities are satisfied for any $C\geq C_P(\mu)$ (resp. $C_{LS}(\mu)$). \end{proposition}
It may seem strange to derive bounds for the Poincar\'e inequality starting from such a bound, but we may hope that starting from an a priori bad bound, the method developed in the previous section will help to improve upon the constant.
\smallskip

We will thus assume that $\mu$ satisfies some $W_1I(2)$ inequality in the terminology of \cite{GLWY}, i.e. we have 
\begin{equation}\label{eqalphetstar}
\alpha(s) = \frac{s^2}{C} \quad \textrm{ so that } \quad \alpha^*(s) = \frac{C}{4} \, s^2 \, .
\end{equation}
\medskip

This yields, using $\alpha^*$, the following values in our previous results:
\begin{enumerate}
\item[(1)] \quad in Theorem \ref{propkappalip}, for $r>2$,
\begin{equation}\label{ratelip1}
\theta \, = \, 2 \left(\mu(\kappa) - \frac C4 \, r \, \parallel \kappa\parallel_c^2\right) \, ,
\end{equation}
\item[(2)] \quad in Theorem \ref{proprholip}, 
\begin{equation}\label{ratelip2}
\theta \, = \, \mu(\rho) - \frac{C}{4} \, \parallel \rho\parallel_c^2 \, ,
\end{equation}
\item[(3)] \quad in Proposition \ref{propPoinclip}, 
\begin{equation}\label{ratelip3}
C_P(\mu) \, \leq \, \frac{1}{\mu(\rho) - \frac{C}{2} \, \parallel \rho\parallel_c^2} \, ,
\end{equation}
\end{enumerate}
provided $\theta >0$ or $\mu(\rho) - \frac{C}{2} \, \parallel \rho\parallel_c^2 > 0$.
\medskip

Similarly, using $\alpha$ instead of $\alpha^*$, we get:
\begin{enumerate}
\item[(1)] \quad in Theorem \ref{propkappalip}, $\theta \, = \, 2 \, \varepsilon \, \mu(\kappa)$ with  $r>2$ satisfying 
\begin{equation}\label{rateosc1}
 r = \frac{1}{C}  \, \frac{1}{\varepsilon \, \mu(\kappa) - \kappa_0} \, \frac{(1-\varepsilon)^2 \mu^2(\kappa)}{\parallel \kappa\parallel_c^2 }\, ,
\end{equation}
\item[(2)] \quad in Theorem \ref{proprholip}, $\theta \, =  \, \varepsilon \, \mu(\rho)$ with
\begin{equation}\label{rateosc2}
1 = \frac{1}{C}  \, \frac{1}{\varepsilon \, \mu(\rho) - \rho_0} \, \frac{(1-\varepsilon)^2 \mu^2(\rho)}{\parallel \rho\parallel_c^2 }\, ,
\end{equation}
\item[(3)] \quad in Proposition \ref{propPoinclip}, $C_P(\mu) \, \leq \, \frac{1}{\varepsilon \, \mu(\rho)}$ with
\begin{equation}\label{rateosc3}
2 = \frac{1}{C}  \, \frac{1}{\varepsilon \, \mu(\rho) - \rho_0} \, \frac{(1-\varepsilon)^2 \mu^2(\rho)}{\parallel \rho\parallel_c^2 } \, ,
\end{equation}
\end{enumerate}
provided $0<\varepsilon<1$.

In the latter cases, denoting by $g$ either $\kappa$ or $\rho$ and by $a$ either $2$ or $1$ we have 
\begin{equation}\label{leqtosolve}
aC(\varepsilon \mu(g)-g_0)\parallel g\parallel_c^2 = (1-\varepsilon)^2 \mu^2(g) \, ,
\end{equation}
 so that denoting by $\eta=\varepsilon \, \mu(g)$, $$h(\eta) \, = \, \eta^2 \, - \, \eta(2\mu(g)+aC\parallel g\parallel_c^2) \, + \, (\mu^2(g)+aCg_0 \parallel g\parallel_c^2) \, = \, 0 \, .$$ Since $h(0)=\mu^2(g)+aCg_0 \parallel g\parallel_c^2$ and $h(\mu(g))=aC\parallel g\parallel_c^2 (g_0-\mu(g)) <0$, we always have a solution $0<\varepsilon<1$, provided $h(0)>0$, solution given by
\begin{equation}\label{eqbestconst}
\varepsilon \, \mu(g) \, = \,  \left(\mu(g) + \frac a2 \, C \, \parallel g\parallel_c^2\right) \left(1 \, - \, \sqrt{1 \, - \, \frac{\mu^2(g)+aCg_0 \, \parallel g\parallel_c^2}{(\mu(g)+\frac a2 \, C \, \parallel g\parallel_c^2)^2}}\right) \, .
\end{equation}
\medskip

One can of course compare the bounds obtained above with $\alpha^*$ and $\alpha$. For simplicity denote by $\beta^*=\mu(g)-\frac{aC}{4} \, \parallel g\parallel_c^2$ the bound obtained with $\alpha^*$ and by $\beta$ the bound obtained with $\alpha$ in \eqref{eqbestconst}. We have:
\begin{enumerate}
\item[(1)] \; if $\beta^*\leq 0$ and $\mu^2(g)+aCg_0 \parallel g\parallel_c^2 \leq 0$, then neither of the bounds are available,
\item[(2)] \; if $\beta^*\leq 0$ and $\mu^2(g)+aCg_0 \parallel g\parallel_c^2 > 0$, then we may use $\beta$,
\item[(3)] \; if $\beta^*> 0$ and $\mu^2(g)+aCg_0 \parallel g\parallel_c^2 \leq 0$, then we may use $\beta^*$,
\item[(4)] \; if $\beta^*> 0$ and $\mu^2(g)+aCg_0 \parallel g\parallel_c^2 > 0$, then $\beta^*\leq \beta$ provided $\mu(g) \leq \frac{5a}{16} \, C \, \parallel g\parallel_c^2 \, + \, g_0$, and the converse inequality holds otherwise.
\end{enumerate} 
The proof of (4) is elementary. These results clearly indicate that both approaches may be of interest.
\medskip

\subsection{The positive curvature case. \\ \\}\label{subsecpos}

Let we illustrate first the previous results when $\rho_0>0$. 

In this case we already know that a lot of inequalities are satisfied, in particular 
\begin{equation}\label{eqsvr}
W_1(P^*_t\nu \, , \, \mu) \, \leq \, e^{- \, \rho_0 \, t} \, W_1(\nu,\mu)
\end{equation}
according to \cite{RS} and 
\begin{equation}\label{eqBEpoinc}
C_P(\mu) \, \leq \, \frac{1}{\rho_0} \quad \textrm{ and } \quad C_{LS}(\mu) \, \leq \, \frac{2}{\rho_0}
\end{equation}
according to the Bakry-Emery criterion. As we have previously mentioned, this statement is rigid, and under a non-trivial variable curvature bound the constants should get strictly better. We shall see what improvements can be obtained using the previous results.
\medskip

Since $\rho_0>0$, \eqref{leqtosolve} always admits a solution for some $\varepsilon \in (0,1)$. In order to compare with the previous bounds one can solve \eqref{leqtosolve} using $\eta = \varepsilon \, \mu(\rho) - \rho_0$, yielding another expression equivalent to \eqref{eqbestconst}
\begin{equation}\label{eqbestconst2}
\varepsilon \, \mu(\rho) \, = \, \rho_0 \, + \, \left((\mu(\rho)-\rho_0) + \frac a2 \, C \, \parallel \rho\parallel_c^2\right) \left(1 \, - \, \sqrt{1 \, - \, \frac{(\mu(\rho)- \rho_0)^2}{((\mu(\rho)-\rho_0)+\frac a2 \, C \, \parallel \rho\parallel_c^2)^2}}\right) \, .
\end{equation}
We have thus obtained the following results (using $\alpha$):

\begin{theorem}\label{BEimplipW1}
Assume that the assumptions of Theorem \ref{proprholip} are satisfied. Assume in addition that $\rho$ is bounded from below by some constant $\rho_0 > 0$. 

Then, for all $\nu$, all $t>\eta>0$, $$W_1(P_t^*\nu\, , \, \mu) \, \leq \, 2 \, \left(\sup_x \, r(2\eta,x,x))^{\frac 12} \right) \; e^{(\theta-\rho_0)\, \eta} \; e^{ -  \, \theta \, t} \; W_1(\nu,\mu) \, ,$$ where $\theta$ is given by one of the following expressions:
\begin{enumerate}
\item[(1)] \; either $$\theta = \, \rho_0 \, + \, \left((\mu(\rho)-\rho_0) + \frac{1}{2 \, \rho_0} \, \Osc^2(\rho)\right) \left(1 \, - \, \sqrt{1 \, - \, \frac{(\mu(\rho)- \rho_0)^2}{((\mu(\rho)-\rho_0)+\frac {1}{2 \, \rho_0} \, \Osc^2(\rho))^2}}\right) \, ,$$
\item[(2)] \; or $$\theta = \, \rho_0 \, + \, \left((\mu(\rho)-\rho_0) +  \, \frac {2}{ \rho^2_0 } \, \parallel\rho\parallel_{Lip}^2\right) \left(1 \, - \, \sqrt{1 \, - \, \frac{(\mu(\rho)- \rho_0)^2}{((\mu(\rho)-\rho_0)+\, \frac {2}{\rho^2_0} \, \parallel\rho\parallel_{Lip}^2)^2}}\right) \, .$$
\end{enumerate}
\end{theorem}
\smallskip

\begin{theorem}\label{BEimpliPoinc}
Similarly, under the assumptions of Theorem \ref{propPoinclip} and if $\rho_0 >0$, $$C_P(\mu) \leq \frac{1}{\rho_0 + \varepsilon'}$$ with 
\begin{enumerate}
\item[(1)] \; either
$$\varepsilon' =  \left((\mu(\rho)-\rho_0) + \, \frac{1}{\rho_0} \, \Osc^2(\rho)\right) \left(1 \, - \, \sqrt{1 \, - \, \frac{(\mu(\rho)- \rho_0)^2}{((\mu(\rho)-\rho_0)+ \, \frac{1}{\rho_0}  \, \Osc^2(\rho))^2}}\right) \, ,$$
\item[(2)] \; or $$\varepsilon' = \left((\mu(\rho)-\rho_0) +  \, \frac {4}{\rho^2_0} \, \parallel\rho\parallel_{Lip}^2\right) \left(1 \, - \, \sqrt{1 \, - \, \frac{(\mu(\rho)- \rho_0)^2}{((\mu(\rho)-\rho_0)+\, \frac {4}{\rho^2_0} \, \parallel\rho\parallel_{Lip}^2)^2}}\right) \, .$$
\end{enumerate}
\end{theorem}
\smallskip

What is interesting here is that one (almost) always obtains improvements of the Bakry-Emery constants, { in agreement with the rigidity phenomenon \cite{CZ17}}. Of course, for the convergence in $W_1$ distance some constant has to appear in front of the exponential rate. Also notice that when $\rho_0 \to 0$ or when $\Osc(\rho)$ or $\parallel \rho\parallel_{Lip} \to + \infty$, $\varepsilon'$ and $\theta$ go to $0$. Also notice that if $\rho$ is constant, its oscillation and its Lipchitz norm vanish so that $\theta$ and $\varepsilon'$ are equal to $0$.

As explained before, in some cases it is possible to obtain better constants by using $\alpha^*$ instead of $\alpha$. Explicit computations are left to the interested reader.

\begin{remark}\label{remultra}
Recall that if the semi-group is \emph{ultra-bounded}, then $\sup_x r(2\eta,x,x)$ is finite. Examples of ultra-bounded semi groups are well known. For example if $V$ is bounded below, goes to infinity at infinity and satisfies $$\frac 12 \, |\nabla V|^2 \, - \, \Delta V \, \geq c \, |V|^\gamma -  d \, \mathbf 1_{|x|\leq R}$$ for some $R>0$, $c,d>0$ and $\gamma>1$, the semi-group is ultra-bounded (see \cite{CatToul} Remark 5.5). With $\gamma=1$ one gets a log-Sobolev inequality, and with $\gamma=0$ a Poincar\'e inequality. Other conditions, such as Foster-Lyapunov conditions, may also be used.
\hfill $\diamondsuit$
\end{remark}
\smallskip

\subsection{Using the log-Sobolev constant. \\ \\}\label{subseclogsob}

In this subsection we will assume that $\mu$ satisfies a log-Sobolev inequality with constant $C_{LS}(\mu)$. In this situation it is well known that it also satisfies a Talagrand inequality $T_2(C_{LS}(\mu))$, so that we have
\begin{equation}\label{eqtala}
W_1(P_t^*\nu\, , \, \mu) \leq W_2(P_t^*\nu\, , \, \mu) \leq \sqrt{C_{LS}(\mu) \, H(P_t^*\nu|\mu)} \leq \sqrt{C_{LS}(\mu)} \; e^{- \, (1/C_{LS}(\mu)) \, t} \; \sqrt{H(\nu,\mu)} \, . 
\end{equation}
According to \cite{BGG} remark 2.8, one can improve on the previous inequalities and show that 
\begin{equation}\label{eqtal2a}
W_1(P_t^*\nu\, , \, \mu) \leq W_2(P_t^*\nu\, , \, \mu) \, \leq \, K(\mu) \; e^{- \, (1/C_{LS}(\mu)) \, t} \; W_2(\nu,\mu) \, .
\end{equation}
However we cannot replace $W_2$ by $W_1$ in the right hand side. According to Theorem \ref{propkappalip}, we can also deduce a similar result, using our method. 
\medskip

But according to what precedes, Theorem \ref{BEimplipW1} extends to the case $\rho_0\leq 0$, yielding $$W_1(P_t^*\nu\, , \, \mu) \leq C(\mu) \, e^{-\theta \, t} \, W_1(\nu,\mu)$$  in the following two situations,
\begin{enumerate}
\item[(1)] \; if the log-Sobolev constant satisfies
\begin{equation}\label{eqlogsobsmall}
C^2_{LS}(\mu) \, < \, \frac{4 \, \mu(\rho)}{\parallel \rho\parallel_{Lip}^2} \, ,
\end{equation}
we may choose $\theta = \mu(\rho) - \frac{C^2_{LS}(\mu)}{4} \, \parallel \rho\parallel_{Lip}^2$,
\item[(2)] \; if the log-Sobolev constant satisfies
\begin{equation}\label{eqlogsobsmall2}
C^2_{LS}(\mu) \, < \, \frac{\mu(\rho)}{|\rho_0| \, \parallel \rho\parallel_{Lip}^2} \, ,
\end{equation}
we may choose 
\begin{equation}\label{eqbestconst1}
\theta \, = \,  \left(\mu(\rho) + \frac 12 \, C^2_{LS}(\mu) \, \parallel \rho\parallel_{Lip}^2\right) \left(1 \, - \, \sqrt{1 \, - \, \frac{\mu^2(\rho)+C_{LS}^2(\mu) \, \rho_0 \, \parallel \rho\parallel_{Lip}^2}{(\mu(\rho)+\frac 12 \, C^2_{LS}(\mu) \, \parallel \rho\parallel_{Lip}^2)^2}}\right) \, .
\end{equation}
\end{enumerate}
\medskip

\subsection{The log-concave case. \\ \\}\label{subseclogconc}

The log-concave case corresponds to $\rho_0=0$. It was shown by Bobkov (\cite{bob99}) that such measures satisfy a Poincar\'e inequality. A simple proof was given in \cite{BBCG}. 

In addition, general quantitative bounds are known for $C_P(\mu)$ in this situation. The easiest one, mentioned in \cite[p.11]{Alonbast} but credited to \cite{KLS} reads as follows
\begin{proposition}\label{propklsvar}
Let $\mu$ be a log-concave distribution of probability. Define $\Var_\mu(x)=\mu(|x-\mu(x)|^2)$ and $\lambda^*$ as the largest eigenvalue of the covariance matrix of $\mu$. Then
\begin{equation}\label{eqkls1}
C_P(\mu) \, \leq \, 4 \, \Var_\mu(x) \, \leq \, 4 \, n \, \lambda^*
\end{equation}
\end{proposition}
It is conjectured (Kannan-Lovasz-Simonovits or K-L-S conjecture) that $$C_P(\mu) \leq C \, \lambda^*$$ for some universal constant $C$ which is thus dimension free. One can look at the original paper \cite{KLS} and at the monograph \cite{Alonbast} for an almost up to date state of the art. The converse inequality $\lambda^* \leq C_P(\mu)$ is always true. The best known result in this direction is due to the work of Lee and Vempala \cite{leevempfoc}: for some universal constant $C$,
\begin{equation}\label{eqleevemp}
C_P(\mu) \, \leq \, C \, n^{\frac 12}  \, \lambda^* \, .
\end{equation}

Except in the gaussian case where all quantities coincide, the relationship between the curvature and $\lambda^*$ is quite unclear, except that, according to the Bakry-Emery criterion $\lambda^* \leq 1/\rho_0$. 
\smallskip

Notice that we may assume that $0=\rho_0=\inf \rho$. Indeed if $\inf \rho >0$ we may always find a smaller modification on a set of small measure with an infimum equal to $0$ and then pass to the limit w.r.t. the measure of this set. Hence we may assume that $\Osc(\rho)=\sup \rho$.

Similarly we may replace $\rho$ by $\rho_K=\min(\rho,K)$, satisfying again $\rho_0=0$ and for a $K$ to be chosen later such that $\mu(\rho_K)>0$. Of course $\mu(\rho_K) \leq \mu(\rho)$. In order to get a lower bound for $\mu(\rho_K)$, it is enough to write $$\mu(\rho_K) \geq K \, \mu(\rho\geq K) \geq \frac 12 \, med_\mu(\rho)$$ if we choose $K=med_\mu(\rho)$. Of course the choice of a median is artificial and we could replace it by any quantile, up to a modification of the constant pre-factor. \medskip

The other specific feature is that, still in the log-concave case, we can replace $\rho$ by $\lambda \, \rho$ for any $0<\lambda \leq 1$. In particular we may choose $\lambda$ in \eqref{ratelip2} in such a way that $$\mu(\rho) - \frac{C \lambda}{4} \, \parallel \rho \parallel_c^2 \, \geq \, \frac{\mu(\rho)}{2} \, .$$ It follows 

\begin{theorem}\label{logconcimplipW1}
Assume that the assumptions of Theorem \ref{proprholip} are satisfied. Assume in addition that $\rho_0 = 0$. 

Then, for all $\nu$, all $t>\eta>0$, $$W_1(P_t^*\nu \, , \, \mu) \, \leq \, 2 \, \left(\sup_x \, r(2\eta,x,x))^{\frac 12} \right) \; e^{\theta \, \eta} \; e^{ -  \, \theta \, t} \; W_1(\nu,\mu) \, ,$$ where $\theta$ is given by one of the following expressions:
\begin{enumerate}
\item[(1)] \quad $\theta = \min \left(\frac{\mu(\rho)}{2} \, ; \, \frac{\mu^2(\rho)}{\parallel \rho \parallel_{Lip}^2 \, C_{LS}^2(\mu)}\right)$ ,
\item[(2)] \quad  $\theta = \min \left(\frac{\mu(\rho)}{2} \, ; \, \frac{\mu^2(\rho)}{\Osc^2(\rho) \,  C_{P}(\mu)}\right)$ \, .
\end{enumerate}
In particular we may always take $\theta = \frac 14 \; \min \left(med_{\mu}(\rho) \, ; \, \frac{1}{C_P(\mu)} \right)$.
\end{theorem}

\begin{remark}\label{remlogconclogsob}
If we assume that the semi-group is ultra-bounded, then it is hyper-bounded so that $\mu$ satisfies a defective log-Sobolev inequality $\Ent_\mu(f^2) \leq C \, \mu(|\nabla f|^2) + A \, \mu(f^2)$. According to the Rothaus lemma, since a log-concave $\mu$ always satisfies a Poincar\'e inequality, the previous defective log-Sobolev inequality can be tightened into a true log-Sobolev inequality with a log-Sobolev constant $C_{LS}(\mu) \leq (C+(A+2)C_P(\mu))$. The difference with the previous subsection is that we do not have to assume that this constant is small enough. In addition the value of the constant $C$  { in the defective LSI approach} depends on the $L^\infty$ norm of the density {of the measure itself}.

Notice that the Theorem applies for instance with $V(x)=|x|^\beta$ with $\beta>2$.
\smallskip

It is tempting to use the same normalization trick in \eqref{ratelip3} with $C=C_P$ in order to get a self improving inequality. Unfortunately this yields a trivial $\Osc^2(\rho)/\mu^2(\rho) \geq 1$, i.e. \eqref{ratelip3} is not self improving. The same disappointing fact, somewhat more tricky to check,  holds with \eqref{eqbestconst}.
\hfill $\diamondsuit$
\end{remark}
\medskip

\begin{remark}\label{remBL}
In the log-concave situation another approach is possible for the Poincar\'e constant, using the famous Brascamp-Lieb inequality
\begin{equation}\label{eqBL}
\Var_\mu(f) \, \leq \, \int \; \langle \nabla f \, , \, (Hess V)^{-1} \, \nabla f\rangle \, d\mu \, .
\end{equation}
In particular it holds 
\begin{equation}\label{eqembl}
\Var_\mu(f) \, \leq \, \left(\int \; \sup_{u \in \mathbb S^{n-1}} \, \langle  u \, , \, (Hess V)^{-1} \, u\rangle \, d\mu\right) \, \parallel |\nabla f|\parallel_\infty^2 \, .
\end{equation}
Using E. Milman's result (see \cite[Theorem 2.4]{emil1} or \cite[Theorem 2.7 and Theorem 2.14]{CGlogconc} for extensions) we deduce that there exists an universal constant $C$ (see \cite{CGlogconc} for explicit bounds) such that 
\begin{equation}\label{eqemb2}
C_P(\mu) \, \leq \, C \, \left(\int \; \sup_{u \in \mathbb S^{n-1}} \, \langle  u \, , \, (Hess V)^{-1} \, u\rangle \, d\mu\right) \, \leq \, C \, \mu(1/\rho) \, .
\end{equation}
We thus have a control of the Poincar\'e constant using some \emph{mean} of the ``inverse'' curvature, {which recovers, up to the value of the constant prefactor, Veysseire's result \cite{Vey11}.} \hfill $\diamondsuit$
\end{remark}
\medskip

\section{Comparing with a direct approach in the ultra-bounded case.}\label{secultrab}

In many of the previous results the quantity $$\sup_x \, r(2\eta,x,x)$$ appears in the pre-factor for the $W_1$ convergence. As we said, finiteness of the above quantity is ensured when the semi-group is ultra bounded, i.e. for all $s>0$, $$\sup_{x,y} \, r(s,x,y) \, \leq \, K(s) \,  < \, +\infty \, .$$ Actually, we do not know examples where the semi-group is not ultra bounded, but $r(2\eta,x,x)$ is bounded in $x$. In the ultra-bounded situation, one may directly control the $W_1$ decay and see that the rate of the exponential decay obtained in the previous two subsections is not optimal.
\medskip

First recall that any Boltzmann probability measure $e^{-V(x)} \, dx$ with $V$ locally bounded satisfies a weak Poincar\'e inequality (see \cite{RW,CatPota}). If the semi-group is ultra-bounded it is hyper-bounded and together with the weak Poincar\'e inequality we get that a (true) Poincar\'e inequality is satisfied (see \cite[Proposition 5.13]{CatPota}), hence according to the Rothaus lemma $\mu$ satisfies a logarithmic Sobolev inequality. In our approach here, we shall only use the Poincar\'e inequality.

Recall that $$W_1(P^*_t\nu \, , \, \mu) \, = \, \sup_{\parallel f\parallel_{Lip}=1} \, \left(P^*_t\nu(f) - \mu(f)\right) \, .$$ Replacing $f$ by $f-f(0)$ does not change the previous expression. Therefore, since we know that $$P^*_\eta \nu(dy)= \left(\int \, r(\eta,x,y) \, \nu(dx)\right) \, \mu(dy) \, = \, r(\eta,\nu, y) \, \mu(dy)$$ and since $P_s$ is symmetric, using the Lipschitz bound on $f$ we get
\begin{equation}\label{eqW1direct1}
W_1(P^*_t\nu \, , \, \mu) \, \leq \, \int \, |x| \, |P_{t-\eta}r(\eta,\nu,.) - 1| \, \mu(dx) \, .
\end{equation}
If $r(\eta,\nu,.) \in \mathbb L^2(\mu)$, using first Cauchy-Schwarz inequality and then the Poincar\'e inequality we get
\begin{eqnarray}\label{eqW1direct2}
W_1(P^*_t\nu \, , \, \mu) &\leq& \mu^{\frac 12}(|x|^2) \; e^{- \, \frac{t-\eta}{C_P(\mu)}} \, \mu^{\frac 12}(|r(\eta,\nu,.)-1|^2) \nonumber \\ &\leq& \mu^{\frac 12}(|x|^2) \; e^{- \, \frac{t-\eta}{C_P(\mu)}} \, \left(\int r(2\eta,x,x) \, \nu(dx) \, - \, 1\right)^{\frac 12} \, .
\end{eqnarray}
It follows that the rate of exponential convergence is $\theta=1/C_P(\mu)$, whatever the curvature is. We are thus { gaining a factor $4$} compared with Theorem \ref{logconcimplipW1}, but we are losing the pre-factor $W_1(\nu,\mu)$ in this direct approach.
\smallskip

If in addition $r(s,.,.)$ is bounded, $$\left(\int r(2\eta,x,x) \, \nu(dx) \, - \, 1\right)^{\frac 12} \, = \, \mu^{\frac 12}(|r(\eta,\nu,.)-1|^2) \, \leq \, (K(\eta)-1)^{\frac 12} \, \parallel \nu \, - \, \mu\parallel_{TV}^{\frac 12} \, .$$ The following can thus be favorably compared with our previous results in some cases
\begin{proposition}\label{propW1ultra}
If the semi-group is ultra-bounded, then for all $\eta >0$, denoting $K(\eta)=\sup_{x,y} r(\eta,x,y)$ we have $$W_1(P^*_t\nu \, , \, \mu) \leq (K(\eta)-1)^{\frac 12} \, \mu^{\frac 12}(|x|^2) \, e^{\eta/C_P(\mu)} \, \parallel \nu \, - \, \mu\parallel_{TV}^{\frac 12} \, e^{- t/C_P(\mu)} \, .$$
\end{proposition}

\underline{\textbf{Acknowledgments:}} This work was supported by Project EFI (ANR-17-CE40-0030) of the French National Research
Agency (ANR). PC and MF were supported by ANR-11-LABX-0040-CIMI within the program ANR-11-IDEX-0002-02. MF was also supported by ANR Project MESA (ANR-18-CE40-006) and by the Partenariat Hubert Curien Procope. We thank Karl-Theodor Sturm for his lectures in Toulouse in the Spring 2019, during the semester program \emph{Calculus of Variations and Probability} of the Labex CIMI. We also thank Matthias Erbar, Benjamin Jourdain, Georg Menz, Julien Reygner and Gabriel Stoltz for valuable discussions.

\bibliographystyle{plain}
\bibliography{CFG}
\end{document}